\documentclass[a4paper]{amsart}
\usepackage{amsmath,amsthm,amssymb,latexsym,epic,bbm,comment,mathrsfs}
\usepackage{graphicx,enumerate,stmaryrd,xcolor}
\usepackage[all,2cell,knot]{xy}
\xyoption{2cell}

\newtheorem{theorem}{Theorem}
\newtheorem{lemma}[theorem]{Lemma}
\newtheorem{corollary}[theorem]{Corollary}
\newtheorem{proposition}[theorem]{Proposition}

\newtheorem{example}[theorem]{Example}

\newtheorem{remark}[theorem]{Remark}
\usepackage[all]{xy}
\usepackage[active]{srcltx}
\usepackage[parfill]{parskip}
\usepackage{enumerate}

\begin{document}

\title[Kronecker coefficients for inverse semigroups]
{Kronecker coefficients for \\ (dual) symmetric inverse semigroups}
\author[V.~Mazorchuk and S.~Srivastava]{Volodymyr Mazorchuk and Shraddha Srivastava}

\begin{abstract}
We study analogues of Kronecker coefficients for symmetric 
inverse semigroups, for dual symmetric inverse semigroups and
for the inverse semigroups of bijections between subquotients 
of finite sets. In all cases we reduce the problem of 
determination of such coefficients to some group-theoretic and
combinatorial problems. For symmetric inverse semigroups,
we provide an explicit formula in terms of the classical 
Kronecker and Littlewood--Richardson coefficients for symmetric
groups.
\end{abstract}

\maketitle

\section{Introduction and description of the results}\label{s1}

Kronecker and Littlewood--Richardson coefficients are numerical 
bookkeeping tools that play important role in the representation
theory of symmetric groups and in the theory of symmetric polynomials.
While Littlewood--Richardson coefficients have transparent 
combinatorial description, to find such a description for Kronecker 
coefficients is a famous open problem, see \cite{Sa} for more details.

Symmetric groups have a number of different analogues in semigroup
theory. The most obvious one is the full transformation semigroup
on a set, that is the semigroup of all endomorphisms of a set.
There are reasons to consider this latter semigroup as ``too difficult'',
for example, in contrast with the group algebra of a (finite) symmetric 
group over the complex numbers, the semigroup algebra of 
this semigroup over the complex
numbers is not semi-simple, see \cite{St}. There are two, slightly 
less obvious but still natural, generalizations of the symmetric group 
which do have the property that the corresponding semigroup 
algebras over the complex numbers are semi-simple.
These are the symmetric inverse and the dual symmetric inverse semigroups.

The symmetric inverse semigroup on a set $X$ is the semigroup of all
bijections between subsets of $X$ under a natural notion of composition of 
such bijections, see \cite{GM}. In the theory of inverse semigroups,
it plays a role similar to the one the symmetric group plays in group theory.
It also appears naturally in other circumstances, for example, in the
theory of $FI$-modules, that is, in the representation theory 
of the category of injections between finite sets, see \cite{CEF}.

The dual symmetric inverse semigroup on a set $X$ is the semigroup of all
bijections between quotients of $X$ under a natural notion of composition of 
such bijections, see \cite{FL}. There are certain categorical reasons to 
use the word dual in this setup, see \cite{FL,KMaz1} for details,
even if there is no obvious underlying duality. This semigroup also
appears naturally in the theory of partition algebras, see \cite{KMaz},
and in the study of factor powers of symmetric groups, see \cite{Maz}.
The symmetric inverse and the dual symmetric inverse semigroups 
are connected by a Schur-Weyl duality, see \cite{KMaz}.
In \cite{JMPSSZ} one can find another transparent connection between 
the symmetric inverse and the dual symmetric inverse semigroup.

In this article, we investigate analogues of Kronecker coefficients
for both, the symmetric inverse semigroup and the dual symmetric inverse
semigroup. In both cases, the problem we address strictly contains
the original problem of determination of Kronecker coefficients for
symmetric groups. This is because the symmetric group is the group
of units (i.e. invertible elements) both for the symmetric inverse and
the dual symmetric inverse semigroups. Therefore our ambition is not 
to provide a complete combinatorial solution, which seems to be 
too difficult, but rather to reduce the problem to some
group-theoretic problem. One could argue that it is natural to 
consider a semigroup-theoretic problem ``solved'' if it is reduced to 
a group-theoretic problem.  Our interest in the Kronecker problem
for symmetric inverse semigroups originates from our attempts 
to understand the techniques developed and used in \cite{BV} 
during our work on \cite{MS}.

The result we obtain in the case of symmetric inverse semigroups
is the most complete one. In this case the formula for 
Kronecker coefficients looks as follows, see Theorem~\ref{thm1}:
\begin{displaymath}
\mathbf{k}_{\lambda,\mu}^\nu= 
\sum_{\alpha\vdash a}\sum_{\beta\vdash b}\sum_{\gamma\vdash c}
\sum_{\delta\vdash b}\sum_{\epsilon\vdash b}\sum_{\kappa\vdash a+b}
\mathbf{c}_{\alpha,\beta}^\lambda
\mathbf{c}_{\gamma,\delta}^\mu
\mathbf{g}_{\beta,\delta}^\epsilon
\mathbf{c}_{\alpha,\epsilon}^\kappa
\mathbf{c}_{\kappa,\gamma}^\nu.
\end{displaymath}
Here $\mathbf{g}_{\beta,\delta}^\epsilon$ are the classical 
Kronecker coefficients for symmetric groups and 
$\mathbf{c}_{\kappa,\gamma}^\nu$ are the classical 
Littlewood--Richardson coefficients for symmetric groups.
The case of symmetric inverse semigroups is considered
and this formula is proved in Section~\ref{s4}.
Along the way we also show that tensor product of 
cell modules for symmetric inverse semigroups have 
a cell filtration and determine the corresponding
multiplicities, see Corollary~\ref{cor3}.

In Section~\ref{s5} we treat the case of the dual symmetric
inverse semigroup. Here the result we obtain, see 
Corollary~\ref{cor-prop7}, is much less explicit and relies
on two inputs, one combinatorial and one group-theoretic.
The combinatorial input is the set of certain 
rectangular $0/1$-matrices, up to permutation of rows
and columns. To each element of this set, we associate
a group-theoretic problem in terms of induction and 
restriction between certain subgroups of symmetric groups,
see Proposition~\ref{prop7}. Our answer is given by 
adding up all solutions. There is a number of interesting 
problems which appear as special cases of our description,
see Examples~\ref{ex8}, \ref{ex9}, \ref{ex9n} and \ref{ex9nn}.
One of them is related to partition algebras and is discussed in 
more detail in Section~\ref{s7}.

In Section~\ref{s6}, we generalize the results of Section~\ref{s5}
to the inverse  semigroup $PI_n^*$ of all bijections between subquotients 
of a set with $n$ elements. Note that both $I_n^*$ and $PI_n^*$
appear as Schur-Weyl duals of the symmetric inverse semigroup,
see \cite[Theorems~1~and~2]{KMaz}. The semigroup $PI_n^*$ is, in some 
sense, a mixture between  $IS_n$ and $I_n^*$. Here our results 
are very similar to results we obtain in Section~\ref{s5}. The main
difference is that the combinatorial input for the solution becomes
more involved and consists of $0/1$-matrices with a fixed special
column and a fixed special row.

Finally, in Section~\ref{s7} we look in more detail into one specific
situation which pops up in Section~\ref{s5}, see Example~\ref{ex8}.
The original problem here is to understand restriction coefficients
for the restriction from $S_{kl}$ to $S_k\times S_l$, where the
latter is naturally embedded into the former as the group of
independent permutations of rows and columns for a Young tableau
of the shape $(k^l)$ (or of the shape $(l^k)$). We use a Schur-Weyl duality
to interpret these restriction coefficients as certain Kronecker-type
coefficients for partition algebras (with different parameters).
We use this interpretation to establish, in Theorem~\ref{s7.5.1},
a stability phenomenon for these coefficients (which is 
also obtained in \cite[Theorem~4.2]{Ryba}
by completely different methods). In the final
subsection we show that, unlike the cases we considered in the
previous sections, tensor product of cell modules over 
partition algebras  does not have a cell filtration, in general.
However, it always has a filtration by standard modules,
see Proposition~\ref{prop7.4.1}.
\vspace{1cm}

\subsection*{Acknowledgments}
The first author is partially supported by the Swedish Research Council.

\section{Preliminaries}\label{s3}

Throughout the paper, we  fix an algebraically closed field $\Bbbk$ 
of characteristic zero. For $n\in\mathbb{Z}_{\geq 0}$, we denote by
$\underline{n}$ the set $\{1,2,\dots,n\}$.  Note that
$\underline{0}:=\varnothing$, the empty set.

\subsection{Symmetric groups}\label{s3.1}

For $n\geq 0$,  consider the symmetric group $S_n$  on $\underline{n}$.
For $\lambda\vdash  n$, denote by $\mathbf{S}^\lambda$ the 
{\em Specht  $\Bbbk[S_n]$-module}
corresponding to $\lambda$. For $\lambda,\mu,\nu\vdash n$, 
we denote by $\mathbf{g}_{\lambda,\mu}^\nu$ the corresponding {\em Kronecker 
coefficient}, that is the multiplicity of $\mathbf{S}^\nu$ in 
$\mathbf{S}^\lambda\otimes_\Bbbk\mathbf{S}^\mu$. The latter is a $\Bbbk[S_n]$-module 
via the diagonal action of $S_n$.

For $\lambda\vdash n$, $\mu\vdash m$ and $\nu\vdash n+m$, 
we denote by $\mathbf{c}_{\lambda,\mu}^\nu$ the corresponding 
{\em Littlewood--Richardson coefficient}, that is the 
multiplicity of $\mathbf{S}^\nu$ in  
$\mathrm{Ind}_{\Bbbk[S_n\times S_m]}^{\Bbbk[S_{n+m}]}
(\mathbf{S}^\lambda\otimes_\Bbbk\mathbf{S}^\mu)$,  where $S_n$ is embedded into
$S_{n+m}$ using the canonical embedding of $\underline{n}$ into $\underline{n+m}$
while $S_m$ is embedded into $S_{n+m}$ via the $+n$-shift of the canonical 
embedding of $\underline{m}$ into $\underline{n+m}$.

We  refer  to \cite{Sa} for all details related to symmetric groups
and their representations.

\subsection{Induction and bimodules}\label{s3.2}

For $k+m=n$, consider the group $G=S_n\times S_k\times S_m$ and its subgroup 
$H$ given by the embedding of $S_k\times S_m$ defined as follows:
each $\sigma\in S_k$ is sent to $(\tilde{\sigma},\sigma,e)$,
where $\tilde{\cdot}:S_k\to S_n$ is the natural embedding of 
$S_k$ into $S_n$ with respect to the first $k$ elements; furthermore,
each $\sigma\in S_m$ is sent to $(\hat{\sigma},e,\sigma)$,
where $\hat{\cdot}:S_m\to S_n$ is the natural embedding of 
$S_m$ into $S_n$ with respect to the last $m$ elements.

The group $G$ acts naturally
on $G/H$ making $\Bbbk[G/H]$ into a $G$-module. We can view
this $G$-module as a $\Bbbk[S_n]$-$\Bbbk[S_k\times S_m]$-bimodule
using the canonical anti-involution $\sigma\mapsto \sigma^{-1}$
on $S_k\times S_m$. Since $\Bbbk[S_k\times S_m]\cong 
\Bbbk[S_k]\otimes_\Bbbk\Bbbk[S_m]$,
given an $S_k$-module $V$ and an $S_m$-module $W$, we have the 
$S_n$-module 
\begin{equation}\label{eq-s3.2-1}
\Bbbk[G/H]\bigotimes_{\Bbbk[S_k\times S_m]} \big(V\otimes_\Bbbk W\big).
\end{equation}
It is easy to see that, mapping $H$ to $e$, gives rise to an isomorphism
between the $\Bbbk[S_n]$-$\Bbbk[S_k\times S_m]$-bimodule
$\Bbbk[G/H]$ and the 
$\Bbbk[S_n]$-$\Bbbk[S_k\times S_m]$-bimodule $\Bbbk[S_n]$.
Therefore the $S_n$-module in \eqref{eq-s3.2-1} is isomorphic to 
$\mathrm{Ind}_{\Bbbk[S_k\times S_m]}^{\Bbbk[S_n]}\big(V\otimes_\Bbbk W\big)$.

\subsection{Tensor product and bimodules}\label{s3.3}

Consider the group $G=S_n\times S_n\times S_n$ and its subgroup 
$H$ given by the diagonal embedding of $S_n$, that is
$\sigma\mapsto (\sigma,\sigma,\sigma)$. Then $G$ acts naturally
on $G/H$ making $\Bbbk[G/H]$ into a $G$-module. We can view
this $G$-module as an $\Bbbk[S_n]$-$\Bbbk[S_n\times S_n]$-bimodule
using the canonical anti-involution $\sigma\mapsto \sigma^{-1}$
on $S_n\times S_n$ (the last two factors). As $\Bbbk[S_n\times S_n]\cong 
\Bbbk[S_n]\otimes_\Bbbk\Bbbk[S_n]$,
given two $S_n$-modules $V$ and $W$,  we have the 
$S_n$-module 
\begin{equation}\label{eq-s3.3-1}
\Bbbk[G/H]\bigotimes_{\Bbbk[S_n\times S_n]} V\otimes_\Bbbk W.
\end{equation}
It is easy to see that, mapping $H$ to $(e,e)$, gives rise to an isomorphism
between the $\Bbbk[S_n]$-$\Bbbk[S_n\times S_n]$-bimodule
$\Bbbk[G/H]$ and the 
$\Bbbk[S_n]$-$\Bbbk[S_n\times S_n]$-bimodule $\Bbbk[S_n\times S_n]$,
where the left action of $S_n$ is diagonal.
Therefore the $S_n$-module in \eqref{eq-s3.3-1}
is isomorphic to $V\otimes_\Bbbk W$ with the usual (diagonal)
action of $S_n$.

\section{Symmetric inverse semigroups}\label{s4}

\subsection{The semigroup  $IS_n$}\label{s4.1}

Consider the symmetric inverse semigroup $IS_n$ of all bijections between
subsets of $\underline{n}$. The semigroup operation in $IS_n$ is a natural 
analogue of composition for partially defined maps. Here is 
an example for $n=5$, where composition is from right to left and
the elements are presented using the two-row notation for (partial) functions:
\begin{displaymath}
\left(\begin{array}{ccc}2&3&5\\1&5&2\end{array}\right) \circ
\left(\begin{array}{cccc}1&2&4&5\\2&4&3&1\end{array}\right)=
\left(\begin{array}{cc}1&4\\1&5\end{array}\right).
\end{displaymath}

For $\sigma\in IS_n$, we denote by
$\mathrm{dom}(\sigma)$ and $\mathrm{cod}(\sigma)$ the {\em domain} and the
{\em codomain} of $\sigma$,  respectively. Recall, see for example 
\cite[Section~4.4]{GM}, the following description of Green's relations
$\mathcal{L}$, $\mathcal{R}$, $\mathcal{H}$ and $\mathcal{D}=\mathcal{J}$
for $IS_n$:
\begin{itemize}
\item $\sigma\mathcal{L}\pi$ if and only if $\mathrm{dom}(\sigma)=\mathrm{dom}(\pi)$,
\item $\sigma\mathcal{R}\pi$ if and only if $\mathrm{cod}(\sigma)=\mathrm{cod}(\pi)$,
\item $\sigma\mathcal{H}\pi$ if and only if $\mathrm{dom}(\sigma)=\mathrm{dom}(\pi)$
and $\mathrm{cod}(\sigma)=\mathrm{cod}(\pi)$,
\item $\sigma\mathcal{J}\pi$ if and only if $|\mathrm{dom}(\sigma)|=|\mathrm{dom}(\pi)|$.
\end{itemize}
The cardinality  of $\mathrm{dom}(\pi)$ is called the {\em rank} of $\pi$ and
denoted $\mathrm{rank}(\pi)$.

For a subset $X\subset\underline{n}$, we denote by $\varepsilon_X$ the identity
on $X$. Then $\{\varepsilon_X :\, X\subset\underline{n}\}$ is the set of all
idempotents in $IS_n$. For $X,Y\subset\underline{n}$, we have 
$\varepsilon_X\varepsilon_Y=\varepsilon_{X\cap Y}$.

\subsection{Cells and simple $\Bbbk[IS_n]$-modules}\label{s4.2}

Let $\mathbf{L}$ be a left cell (i.e. an $\mathcal{L}$-equivalence class) in $IS_n$.
The linearization $\Bbbk[\mathbf{L}]$ is naturally a left $\Bbbk[IS_n]$-module with the 
action given as follows, for $\sigma\in IS_n$ and $\pi\in \mathbf{L}$:
\begin{equation}\label{eq-nnn}
\sigma\cdot\pi=
\begin{cases}
\sigma\pi,& \text{ if } \sigma\pi\in \mathbf{L};\\
0,& \text{ otherwise.}
\end{cases}
\end{equation}
This is the {\em cell $\Bbbk[IS_n]$-module} corresponding to $\mathbf{L}$.

The set $\mathbf{L}$ contains a unique idempotent, namely 
$\varepsilon_X$, where $X$ is the common domain for all elements in $\mathbf{L}$. 
Let $G(\mathbf{L})$ be the corresponding maximal subgroup in $IS_n$, 
i.e. the $\mathcal{H}$-equivalence class of $\varepsilon_X$. 
Set $k:=|X|$. Then $G(\mathbf{L})=S_X\cong S_{k}$.
We have the vector space decomposition
\begin{equation}\label{eq-s4.2-1}
\Bbbk[\mathbf{L}]\cong 
\bigoplus_{Y\subset \underline{n}, |Y|=k}\varepsilon_Y \Bbbk[\mathbf{L}].
\end{equation}
The group $G(\mathbf{L})\cong S_k$ acts on $\mathbf{L}$ by multiplication on the right.
This equips $\Bbbk[\mathbf{L}]$ with the natural structure of a
$\Bbbk[IS_n]$-$\Bbbk[S_k]$-bimodule. In fact, each 
$\varepsilon_Y \Bbbk[\mathbf{L}]$ in \eqref{eq-s4.2-1} is a free right
$\Bbbk[S_k]$-module of rank one.
Up to isomorphism, the bimodule $\Bbbk[\mathbf{L}]$ does not depend on the choice of 
$\mathbf{L}$ inside its $\mathcal{J}$-cell (i.e. for the fixed $k$).

For any $\lambda\vdash k$, the $\Bbbk[IS_n]$-module
\begin{equation}\label{eq-s.4.2-mm}
\mathbf{N}^{\lambda}:=\Bbbk[\mathbf{L}]\bigotimes_{\Bbbk[S_k]}
\mathbf{S}^\lambda
\end{equation}
is simple. By the above, up to isomorphism, it does not depend on the choice of 
$\mathbf{L}$. Moreover, up to isomorphism,
each simple $\Bbbk[IS_n]$-module has such a form,  for some $k$ and $\lambda$ as above.
Putting all this together,  we see that the isomorphism classes of simple
$\Bbbk[IS_n]$-modules are in bijection with partitions $\lambda\vdash k$, where
$k\leq n$.

We refer to \cite[Chapter~11]{GM}  and \cite{St} for further details.

\subsection{Kronecker coefficients for $IS_n$}\label{s4.3}

Being a semigroup algebra, $\Bbbk[IS_n]$ has the canonical comultiplication given
by the diagonal map $\sigma\mapsto \sigma\otimes\sigma$. Therefore, for two
$\Bbbk[IS_n]$-modules $M$ and $N$, the tensor  product $M\otimes_\Bbbk N$
has the natural  structure of a $\Bbbk[IS_n]$-module.

For a fixed $n$, let $k,l,m\leq n$, $\lambda\vdash k$,  $\mu\vdash l$
and $\nu\vdash m$. We denote by $\mathbf{k}_{\lambda,\mu}^\nu$ 
the multiplicity of $\mathbf{N}^\nu$ in 
$\mathbf{N}^\lambda\otimes_\Bbbk\mathbf{N}^\mu$. It is natural to call
these $\mathbf{k}_{\lambda,\mu}^\nu$ the {\em Kronecker coefficients for
$IS_n$}. Our main result in this section is the following:

\begin{theorem}\label{thm1}
\vspace{2mm}

\begin{enumerate}[$($a$)$]
\item\label{thm1.1} If $m<\max\{k,l\}$ or $m>k+l$, then $\mathbf{k}_{\lambda,\mu}^\nu=0$.
\item\label{thm1.2} If $\max\{k,l\}\leq m\leq k+l$, set $b:=k+l-m$, $a:=k-b$ and $c:=l-b$.
Then:
\begin{equation}\label{eq2}
\mathbf{k}_{\lambda,\mu}^\nu= 
\sum_{\alpha\vdash a}\sum_{\beta\vdash b}\sum_{\gamma\vdash c}
\sum_{\delta\vdash b}\sum_{\epsilon\vdash b}\sum_{\kappa\vdash a+b}
\mathbf{c}_{\alpha,\beta}^\lambda
\mathbf{c}_{\gamma,\delta}^\mu
\mathbf{g}_{\beta,\delta}^\epsilon
\mathbf{c}_{\alpha,\epsilon}^\kappa
\mathbf{c}_{\kappa,\gamma}^\nu.
\end{equation}
\end{enumerate}
\end{theorem}

We note that the Kronecker coefficients for $IS_n$ appear in the
theory of FI-modules, see  \cite{CEF}, especially Formula~(17) in \cite{CEF}, 
for details. For $\lambda=(1)$, the computation of $\mathbf{k}^{\nu}_{\lambda,\mu}$ 
can be obtained from \cite[Corollary~2.7]{MS21}.

\subsection{Tensor product of cell  modules}\label{s4.4}

Before we can prove Theorem~\ref{thm1}, we need to consider a simpler 
decomposition problem, namely, that for the tensor product of cell modules.
For $0\leq r\leq n$, we denote by $\mathbf{L}_r$ the left cell 
in $IS_n$ containing  $\varepsilon_{\underline{r}}$.

Let $k$, $l$ and $m$ be as in Theorem~\ref{thm1}. Consider the  $IS_n$-module
$\Bbbk[\mathbf{L}_k]\otimes_\Bbbk\Bbbk[\mathbf{L}_l]$. We are going to show
that this module has a filtration whose subquotients are isomorphic to  cell
modules and we will also determine  the multiplicity of $\Bbbk[\mathbf{L}_m]$
as a subquotient in that filtration. The fact that this multiplicity
is well-defined (i.e. does not depend on the choice of a filtration with cell
subquotients) is clear since non-isomorphic cell modules do not have any common
non-zero simple summands.

The elements in $\mathbf{L}_k$ are given by pairs $(X,\sigma)$, where 
$X$  is a cardinality $k$ subset of $\underline{n}$ and $\sigma$
is a bijection from $\underline{k}$ to $X$. Similarly, 
the elements in $\mathbf{L}_l$ are given by pairs $(Y,\pi)$, where 
$Y$  is a cardinality $l$ subset of $\underline{n}$ and $\pi$
is a bijection from $\underline{l}$ to $Y$. For such a $\xi=(X,\sigma)$, 
the rank of $\varepsilon_{\underline{m}}\xi$ equals $k$ if and only if $X\subset \underline{m}$.
Similarly, for such an $\eta=(Y,\pi)$, the rank of $\varepsilon_{\underline{m}}\eta$ 
equals $l$ if and only if $Y\subset \underline{m}$. Hence
$\varepsilon_{\underline{m}}$ does not kill $\xi\otimes\eta$ if and only if 
$X\subset \underline{m}$ and $Y\subset \underline{m}$
(in particular, $m\geq \max\{k,l\}$).

If we assume  that $X\cup Y\subsetneq \underline{m}$, then 
$\xi\otimes\eta$ is not killed by the idempotent
$\varepsilon_{X\cup Y}$
whose rank is strictly smaller than $m$. This means that such $\xi\otimes\eta$
cannot contribute to the cell module $\Bbbk[\mathbf{L}_m]$. Therefore,
we only need to consider the case $X\cup Y=\underline{m}$
(which implies  $\max\{k,l\}\leq m\leq k+l$). In this case
$|X\cap Y|=k+l-m=b$.

Denote by $Q$ the set of all $\xi\otimes\eta$ as in the previous 
paragraph. The group $S_m$, identified as the maximal subgroup
$G(\mathbf{L}_m)$, for the idempotent
$\varepsilon_{\underline{m}}$, acts on $\xi\otimes\eta\in Q$ 
naturally by left multiplication. The group $S_k\times S_l$ acts on 
$\xi\otimes\eta\in Q$ on the right, where $S_k$ acts on $\sigma$
by right multiplication while $S_l$ acts on $\pi$
by right multiplication. The two actions obviously commute.

Let us determine the cardinality of $Q$.
We can choose $X\cap Y$ in $\binom{m}{b}$ different ways
and then the rest of $X$ in $\binom{m-b}{a}=\binom{m-b}{c}$ different way.
As $\sigma$ and $\pi$ are arbitrary and $a+b+c=m$, we have 
\begin{equation}\label{eq1}
\binom{a+b+c}{b}\binom{a+c}{a}(a+b)!(b+c)!=
\frac{(a+b+c)!(a+b)!(b+c)!}{a!b!c!}
\end{equation}
different choices in total.

Now we can perform the following construction, similar to the ones
presented in Subsections~\ref{s3.2} and \ref{s3.3}:
Consider the group $G=S_{a+b+c}\times S_{a+b}\times  S_{c+b}$.
We split the $a+b+c$ elements on which $S_{a+b+c}$ acts into
$a$ apricot, $b$ blue and $c$ crimson elements. Similarly, 
we split the $a+b$ elements on which $S_{a+b}$ acts into
$a$ apricot, and $b$ blue elements. Finally, 
we split the $b+c$ elements on which $S_{b+c}$ acts into
$b$ blue and $c$ crimson elements. Fix some identification of
apricot elements for $S_{a+b+c}$ with apricot elements for $S_{a+b}$,
and similarly for blue and crimson elements.

Let $H$ be the subgroup of $G$ generated by the following elements:
\begin{itemize}
\item all $(\rho,\rho,e)$, where $\rho$ is a permutation of apricot elements;
\item all $(\rho,\rho,\rho)$, where $\rho$ is a permutation of blue elements;
\item all $(\rho,e,\rho)$, where $\rho$ is a permutation of crimson elements.
\end{itemize}
Clearly, $H$ is isomorphic to $S_a\times S_b\times S_c$.
Consider the linearization $\Bbbk[G/H]$, which is naturally a $G$-module,
in particular, it is a $\Bbbk[S_{a+b+c}]$-$\Bbbk[S_{a+b}\times  S_{c+b}]$-bimodule,
using the canonical anti-involution $\sigma\mapsto\sigma^{-1}$ on $S_{a+b}\times  S_{c+b}$.
Now we can compare the $\Bbbk[S_{a+b+c}]$-$\Bbbk[S_{a+b}\times  S_{c+b}]$-bimodules
$\Bbbk[G/H]$ and $\Bbbk Q$.

\begin{lemma}\label{lem2}
The $\Bbbk[S_{a+b+c}]$-$\Bbbk[S_{a+b}\times  S_{c+b}]$-bimodules
$\Bbbk[G/H]$ and $\Bbbk Q$ are isomorphic.
\end{lemma}

\begin{proof}
Consider $\eta= (\{a+1,\ldots, a+l\}, \tau)$ where $\tau$ is the 
order preserving bijection from $\underline{l}$ to $\{a+1,\ldots,a+l\}$.
Then it is easy to see that the $G$-stabilizer of 
$\varepsilon_{\underline{k}}\otimes\eta\in Q$ equals $H$.
Therefore, sending $H$ to 
$\varepsilon_{\underline{k}}\otimes\eta$  
gives rise to a homomorphism from  $\Bbbk[G/H]$ to $\Bbbk Q$.
This homomorphism is easily seen to be an isomorphism since
the two spaces have the same dimension and 
$\varepsilon_{\underline{k}}\otimes\eta$
generates $\Bbbk Q$, as a $\Bbbk[S_{a+b+c}]$-$\Bbbk[S_{a+b}\times  S_{c+b}]$-bimodule.
\end{proof}

From the definitions, it is easy to see that, for any conjugate $H'$ of $H$ in $G$, 
we have $H'\cap S_{a+b+c} = e$. This implies that, as an $S_{a+b+c}$-module, the 
module $\Bbbk[G/H]$ is free. By Lemma~\ref{lem2} we thus also have that
$\Bbbk Q$ is free, as an $S_{a+b+c}$-module. In particular,
$Q$ is just a disjoint union
of copies of regular $S_m$-orbits. Combining the definitions with
decomposition \eqref{eq-s4.2-1}, we have that each such regular $S_m$-orbit
determines in  $\Bbbk[\mathbf{L}_k]\otimes_\Bbbk\Bbbk[\mathbf{L}_l]$
a subquotient isomorphic to $\Bbbk[\mathbf{L}_m]$.  This, on the one hand,
implies  that $\Bbbk[\mathbf{L}_k]\otimes_\Bbbk\Bbbk[\mathbf{L}_l]$
has a filtration whose subquotients are cell  modules.
On the other hand, it also  says that
the multiplicity of $\Bbbk[\mathbf{L}_m]$ in 
$\Bbbk[\mathbf{L}_k]\otimes_\Bbbk\Bbbk[\mathbf{L}_l]$ equals
the number  of $S_m$-orbits on $Q$.
Summing everything up, we obtain:

\begin{corollary}\label{cor3}
For $0\leq k,l\leq n$, the module $\Bbbk[\mathbf{L}_k]\otimes_\Bbbk\Bbbk[\mathbf{L}_l]$
has a filtration whose subquotients are isomorphic to cell modules.
Moreover, for  $\max\{k,l\}\leq m\leq k+l$,
the multiplicity of $\Bbbk[\mathbf{L}_m]$  in 
$\Bbbk[\mathbf{L}_k]\otimes_\Bbbk\Bbbk[\mathbf{L}_l]$ equals
$\frac{(a+b)!(b+c)!}{a!b!c!}$.
\end{corollary}

\subsection{Proof of Theorem~\ref{thm1}}\label{s4.9}

Claim~\eqref{thm1.1} of Theorem~\ref{thm1} follows from the two estimates
$m\geq \max\{k,l\}$ and $m\leq k+l$, obtained in the previous subsection.

In the previous subsection, we considered the bimodule $\Bbbk Q$ which 
describes the part of the tensor product
$\Bbbk[\mathbf{L}_k]\otimes_\Bbbk\Bbbk[\mathbf{L}_l]$
that determines the $\Bbbk[\mathbf{L}_m]$-summands of that
tensor product. In particular, combining the definition of
$Q$ with \eqref{eq-s.4.2-mm}, it follows that
\begin{equation}\label{eq-s4.9-1}
[\mathbf{N}^\lambda\otimes_\Bbbk\mathbf{N}^\mu:\mathbf{N}^\nu]=
[\Bbbk Q\otimes_{\Bbbk[S_k]\otimes_\Bbbk \Bbbk[S_l]}
\big(\mathbf{S}^\lambda\otimes_\Bbbk\mathbf{S}^\mu\big):\mathbf{S}^\nu].
\end{equation}
Therefore, due to Lemma~\ref{lem2},
to prove Claim~\eqref{thm1.2}, we need  to  show that Formula~\eqref{eq2}
gives the multiplicity
\begin{displaymath}
\big[\Bbbk[G/H]\otimes_{\Bbbk[S_k]\otimes_\Bbbk \Bbbk[S_l]}
\big(\mathbf{S}^\lambda\otimes_\Bbbk\mathbf{S}^\mu\big):\mathbf{S}^\nu\big].
\end{displaymath}

Define an additive functor $\mathcal{F}$ from
$\Bbbk[S_{a+b}\times S_{b+c}]$-mod to $\Bbbk[S_{a+b+c}]$-mod,
for two modules $V\in S_{a+b}$-mod and $W\in S_{b+c}$-mod,
as follows:
\begin{itemize}
\item First, restrict $V$ to $S_a\times S_b$ and
$W$ to $S_b\times S_c$.
\item Next, consider $V\otimes_\Bbbk W$ as an
$S_a\times S_b\times S_c$-module with respect to the 
obvious actions of  $S_a$ and $S_c$ on the first and
on the second components, respectively, 
and the diagonal action of $S_b$.
\item Finally, induce the obtained $S_a\times S_b\times S_c$-module
up to $S_{a+b+c}$.
\end{itemize}

Let us now compute $\mathcal{F}(\mathbf{S}^\lambda\otimes_\Bbbk \mathbf{S}^\mu)$.
To start with, we have:
\begin{displaymath}
\mathrm{Res}_{\Bbbk[S_a\times S_b]}^{\Bbbk[S_{a+b}]}\mathbf{S}^\lambda\cong
\bigoplus_{\alpha\vdash a}
\bigoplus_{\beta\vdash b}
(\mathbf{S}^\alpha\otimes_\Bbbk \mathbf{S}^\beta)^{\oplus 
\mathbf{c}_{\alpha,\beta}^\lambda}
\end{displaymath}
and
\begin{displaymath}
\mathrm{Res}_{\Bbbk[S_b\times S_c]}^{\Bbbk[S_{b+c}]}\mathbf{S}^\mu\cong
\bigoplus_{\gamma\vdash c}
\bigoplus_{\delta\vdash b}
(\mathbf{S}^\delta\otimes_\Bbbk \mathbf{S}^\gamma)^{\oplus 
\mathbf{c}_{\delta,\gamma}^\mu}.
\end{displaymath}
At the next step, we have
\begin{displaymath}
(\mathbf{S}^\alpha\otimes_\Bbbk \mathbf{S}^\beta)\otimes_\Bbbk
(\mathbf{S}^\delta\otimes_\Bbbk \mathbf{S}^\gamma)\cong
\bigoplus_{\epsilon\vdash b}
(\mathbf{S}^\alpha\otimes_\Bbbk \mathbf{S}^\epsilon
\otimes_\Bbbk \mathbf{S}^\gamma)^{\oplus \mathbf{g}_{\beta,\delta}^\epsilon}.
\end{displaymath}
Finally, at the last step, which we realize as the composition of
first inducing from $S_a\times S_b\times S_c$ to $S_{a+b}\times S_c$
and then from the latter to $S_{a+b+c}$, we have 
\begin{displaymath}
\mathrm{Ind}_{S_a\times S_b\times S_c}^{S_{a+b+c}} 
\mathbf{S}^\alpha\otimes_\Bbbk \mathbf{S}^\epsilon
\otimes_\Bbbk \mathbf{S}^\gamma\cong
\bigoplus_{\nu\vdash a+b+c}
\bigoplus_{\kappa\vdash a+b}
\big(\mathbf{S}^\nu\big)^{\oplus \mathbf{c}_{\alpha,\epsilon}^\kappa\cdot
\mathbf{c}_{\kappa,\gamma}^\nu}.
\end{displaymath}
Collecting all these formulae together, we see that the multiplicity
of $\mathbf{S}^{\nu}$ in 
$\mathcal{F}(\mathbf{S}^\lambda\otimes_\Bbbk \mathbf{S}^\mu)$
is given by Formula~\eqref{eq2}.

Let us now try to understand the 
$\Bbbk[S_{a+b+c}]$-$\Bbbk[S_{a+b}\times S_{b+c}]$-bimodule
which realizes $\mathcal{F}$. Taking into account 
the constructions presented in Subsections~\ref{s3.2} and \ref{s3.3},
the combination of the first two steps is given by 
\begin{displaymath}
\Bbbk[S_{a+b}]\otimes_\Bbbk \Bbbk[S_{b+c}]
\end{displaymath}
where the right action is regular, the left actions of 
$S_a$ and $S_c$ are the obvious ones, and the left action of 
$S_b$ is the diagonal action. At the last step, we just tensor with
the regular $\Bbbk[S_{a+b+c}]$-$\Bbbk[S_{a+b+c}]$-bimodule
where the right action is restricted to $S_a\times S_b\times S_c$.

Directly from the definitions, we see that the stabilizer
(in $G$) of the element $e\otimes (e\otimes e)$ from
\begin{equation}\label{eq3}
\Bbbk[S_{a+b+c}]\otimes_{\Bbbk[S_a\times S_b\times S_c]}
\big(\Bbbk[S_{a+b}]\otimes_\Bbbk \Bbbk[S_{b+c}]\big),
\end{equation}
contains $H$. Thus, sending $H$ to $e\otimes (e\otimes e)$, gives
rise to a $\Bbbk[S_{a+b+c}]$-$\Bbbk[S_{a+b}\times S_{b+c}]$-bimodule
homomorphism from $\Bbbk[G/H]$ to the bimodule in \eqref{eq3}.
As $e\otimes (e\otimes e)$ generates the latter, this homomorphism
is surjective. By comparing the dimensions, we see that it is 
bijective. This completes the proof of 
Claim~\eqref{thm1.2} in Theorem~\ref{thm1}.

\subsection{Advance remark}\label{s4.95}

We complete this section with a remark about symmetric  inverse
semigroups which is inspired by Example~\ref{ex8} below
(this example will also  be studied in more detail in 
Section~\ref{s7}).

If we interpret the elements in $IS_n$ as rook $0/1$-matrices,
then the Kronecker product of matrices defines a homomorphism
from $IS_n\times IS_m$ to $IS_{nm}$. Note that this homomorphism
is not injective (as tensoring anything with the zero matrix
outputs the zero matrix). Nevertheless, it gives rise to
a homomorphism of unital associative algebras 
$\Bbbk[IS_n]\otimes_\Bbbk\Bbbk[IS_m]\to\Bbbk[IS_{nm}]$.
In particular, we have the corresponding pullback
functor, which allows us to consider each 
$IS_{nm}$-module as an $IS_n\times IS_m$-module.

Cell modules over symmetric inverse semigroups (and their
direct products) have a special feature: the cell module
for a left cell $\mathbf{L}$ is the linearization of
the obvious transitive action of our semigroup on $\mathbf{L}$
by partial transformations given by left multiplication,
cf. Formula~\eqref{eq-nnn}. Classification of 
transitive acts of inverse semigroups is a classical 
part of the theory of inverse semigroup, see \cite[Theorem~5.4]{La}, and
it is usually formulated in terms of closed 
inverse subsemigroups, which serve as stabilizers of
points. Each such closed inverse subsemigroup 
has a distinguished idempotent and contains
some subgroup $H$ of the maximal  subgroup  corresponding to this idempotent.
Abusing terminology, we can call this $H$
the {\em apex} of  the stabilizer.
A left cell act corresponds exactly to the case
when $H$ is trivial,  i.e. consists only of the 
identity (i.e. only of the corresponding idempotent).

If we now take a left cell $IS_{nm}$-act and pull it
back to $IS_n\times IS_m$,  we can consider a
natural filtration of the latter act  whose subquotients
are transitive subquotient acts. The property that the
apex of each point of the original  act is trivial
is clearly inherited under pullback. Consequently,
the pullback of any cell $IS_{nm}$-module has a
filtration whose subquotients are cell
$IS_n\times IS_m$-modules. One can compare this
with the results of Subsection~\ref{s4.4}.

\section{Dual symmetric inverse semigroups}\label{s5}

\subsection{The semigroup  $I_n^*$}\label{s5.1}

Consider the dual symmetric inverse semigroup $I_n^*$, see
\cite{FL}. The elements of $I_n^*$ are all possible bijections
between the quotients of $\underline{n}$ and the operation is
the natural version of composition for such bijections. Here is an
example for $n=5$, where composition is from right to left and
the elements are presented using the two-row notation for functions:
\begin{displaymath}
\left(\begin{array}{ccc}\{1\}&\{2,3\}&\{4,5\}\\\{1,2\}&\{3,4\}&\{5\}
\end{array}\right) \circ
\left(\begin{array}{ccc}\{1,2\}&\{3\}&\{4,5\}\\
\{1\}&\{2,3,4\}&\{5\}\end{array}\right) =
\left(\begin{array}{cc}\{1,2\}&\{3,4,5\}\\\{1,2\}&\{3,4,5\}\end{array}\right) 
\end{displaymath}

Each element $\xi$ in $I_n^*$ is thus uniquely given by the following
data:
\begin{itemize}
\item a set partition $\xi_d$ of $\underline{n}$, called 
the {\em domain} of $\xi$ and denoted $\mathrm{dom}(\xi)$;
\item a set partition $\xi_r$ of $\underline{n}$
(with the same number of parts as $\xi_d$), called 
the {\em codomain} of $\xi$ and denoted $\mathrm{cod}(\xi)$;
\item a bijection $\overline{\xi}$ from $\underline{n}/\xi_d$ to
$\underline{n}/\xi_r$.
\end{itemize}
The cardinality of $\underline{n}/\xi_d$ is called the 
{\em rank} of $\xi$ and denoted $\mathrm{rank}(\xi)$.
We also call the same number the {\em rank} of $\xi_d$.

Let $\xi,\zeta\in I_n^*$. Recall, see for example 
\cite[Theorem~2.2]{FL}, the following description of Green's relations
$\mathcal{L}$, $\mathcal{R}$, $\mathcal{H}$ and $\mathcal{D}=\mathcal{J}$
for $I_n^*$:
\begin{itemize}
\item $\xi\mathcal{L}\zeta$ if and only if $\mathrm{dom}(\xi)=\mathrm{dom}(\zeta)$,
\item $\xi\mathcal{R}\zeta$ if and only if $\mathrm{cod}(\xi)=\mathrm{cod}(\zeta)$,
\item $\xi\mathcal{H}\zeta$ if and only if $\mathrm{dom}(\xi)=\mathrm{dom}(\zeta)$
and $\mathrm{cod}(\xi)=\mathrm{cod}(\zeta)$,
\item $\xi\mathcal{J}\zeta$ if and only if $\mathrm{rank}(\xi)=\mathrm{rank}(\zeta)$.
\end{itemize}

The idempotents of $I_n^*$ are naturally identified with equivalence
relations on $\underline{n}$: given an equivalence relation
$\rho$ on $\underline{n}$, the corresponding idempotent $\varepsilon_\rho$
is the identity map on $\underline{n}/\rho$.
The maximal subgroup of $I_n^*$ corresponding to $\rho$ is the 
symmetric group $S_\rho$ on the equivalence classes of $\rho$.

The semigroup $I_k^*$ appears naturally in the context of the 
Schur-Weyl  duality for the  $k$-th tensor 
power of the natural representation of $IS_n$, see \cite[Theorem~1]{KMaz}.

\subsection{Cells and simple $I_n^*$-modules}\label{s5.2}

Let $\rho$ be an equivalence relation on $\underline{n}$.
Let $\mathbf{L}_\rho$ be the left cell (i.e. an $\mathcal{L}$-equivalence class) 
in  $I_n^*$ containing $\varepsilon_\rho$.

The linearization $\Bbbk[\mathbf{L}_\rho]$ is naturally a left $I_n^*$-module 
where the  action is given as follows, for $\xi\in I_n^*$ and 
$\zeta\in \mathbf{L}_\rho$:
\begin{displaymath}
\xi\cdot\zeta=
\begin{cases}
\xi\zeta,& \text{ if } \xi\zeta\in \mathbf{L}_\rho;\\
0,& \text{ otherwise.}
\end{cases}
\end{displaymath}
This is the {\em cell $I_n^*$-module} corresponding to $\mathbf{L}_\rho$.

The maximal subgroup $S_\rho$ acts on $\mathbf{L}_\rho$ by multiplication 
on the right. This equips $\Bbbk[\mathbf{L}_\rho]$ with the natural 
structure of an $\Bbbk[I_n^*]$-$\Bbbk[S_\rho]$-bimodule.

For any $\lambda\vdash \mathrm{rank}(\varepsilon_\rho)$, 
the $\Bbbk[I_n^*]$-module
\begin{displaymath}
\mathbf{N}^{\lambda}:=\Bbbk[\mathbf{L}_\rho]\bigotimes_{\Bbbk[S_\rho]}
\mathbf{S}^\lambda
\end{displaymath}
is simple. Up to isomorphism, it does not depend on the choice of 
$\mathbf{L}_\rho$ inside its $\mathcal{J}$-cell (i.e. for the fixed rank of 
$\varepsilon_\rho$). Moreover, up to isomorphism,
each simple $\Bbbk[I_n^*]$-module has such a form,  for some 
$\rho$ and $\lambda$. We refer to \cite{FL}  and \cite{St} for details.

\subsection{Tensor product of cell  modules}\label{s5.3}

Let $\rho$ and $\sigma$ be two set partitions 
(a.k.a. equivalence relations) of $\underline{n}$.
The tensor product
\begin{displaymath}
\Bbbk[\mathbf{L}_\rho]\otimes_\Bbbk\Bbbk[\mathbf{L}_\sigma]
\end{displaymath}
is an $I_n^*$-module via the diagonal action of $I_n^*$. 
In this subsection we are going to show that 
this module has a filtration whose subquotients are isomorphic to cell
modules. Also, we are going to describe a combinatorial object which
determines the corresponding multiplicities (just like in the case
of $IS_n$, the fact that these multiplicities are well-defined
follows directly from the semi-simplicity of the semigroup algebra
of $I_n^*$ combined with the fact that two non-isomorphic cell modules
do not have common isomorphic simple summands).

Set $k=\mathrm{rank}(\varepsilon_\rho)$ and 
$l=\mathrm{rank}(\varepsilon_\sigma)$. 
Let $\tilde{\rho}$ be a set partition of $\underline{n}$ of
rank $k$ and $\tilde{\sigma}$ be a set partition of $\underline{n}$ 
of rank $l$. Consider the set partition $\tau$
of $\underline{n}$ whose parts are exactly the non-empty 
intersections of a part of $\tilde{\rho}$ with a part of $\tilde{\sigma}$.
We denote this $\tau$ by $\tilde{\rho}\cap\tilde{\sigma}$.

\begin{example}\label{ex4}
{\em
If $n=4$, $\tilde{\rho}=\{1,2\}\cup\{3,4\}$ and
$\tilde{\sigma}=\{1,2,3\}\cup\{4\}$, then we have
$\tau=\{1,2\}\cup\{3\}\cup\{4\}$.
}
\end{example}

Let $\preceq$ be the natural order on set partitions: we have
$\alpha\preceq \beta$ if each part of $\alpha$ is a subset of some
part of $\beta$. Let $\xi\in\mathbf{L}_\rho$ and $\zeta\in\mathbf{L}_\sigma$
be such that $\mathrm{cod}(\xi)=\tilde{\rho}$ and 
$\mathrm{cod}(\zeta)=\tilde{\sigma}$. Then $\tau$ is the maximum
element in the set of all set partitions $\omega$ of $\underline{n}$
such that $\varepsilon_\omega\xi\in \mathbf{L}_\rho$
and $\varepsilon_\omega\zeta\in\mathbf{L}_\sigma$.
Note also that the action of $S_\tau$ on $\xi\otimes \zeta$ is, clearly, free.

All this means exactly that the element
$\xi\otimes \zeta$ contributes to a cell subquotient
$\Bbbk[\mathbf{L}_\tau]$ of 
$\Bbbk[\mathbf{L}_\rho]\otimes_\Bbbk\Bbbk[\mathbf{L}_\sigma]$.

To provide a combinatorial object for counting the multiplicities,
we need to invert our problem. Given some $\tau$, we need to find
the number of all possible $\tilde{\rho}$ and $\tilde{\sigma}$
as above which are connected to $\tau$ also in the above sense.

Set $m:=\mathrm{rank}(\varepsilon_\tau)$.
Let $Q_{k,l}^{m}$ be the set of all $k\times l$ matrices with 
$0/1$-entries satisfying the following conditions:
\begin{itemize}
\item exactly $m$ entries are equal to $1$;
\item all rows are non-zero;
\item all columns are non-zero.
\end{itemize}
The group $G:=S_k\times S_l$ acts on $Q_{k,l}^m$
by (independent) permutation of rows and columns. Consider
the corresponding orbit set $Q_{k,l}^m/G$ and set
$\mathbf{p}_{k,l}^m:=|Q_{k,l}^m|$ and
$\mathbf{q}_{k,l}^m:=|Q_{k,l}^m/G|$. Note that,
for $\mathbf{p}_{k,l}^m$ or $\mathbf{q}_{k,l}^m$ to be non-zero, 
we must have the inequalities $\max\{k,l\}\leq m\leq kl$. 
The cardinalities of stabilizers in $G$ of elements of a fixed orbit are equal
(as these stabilizers are conjugate) and,
for an orbit $\mathcal{O}\in Q_{k,l}^m/G$, we denote this  cardinality by 
$|\text{stab}_\mathcal{O}|$.
%

\begin{proposition}\label{prop5}
The $I_n^*$-module $\Bbbk[\mathbf{L}_\rho]\otimes_\Bbbk\Bbbk[\mathbf{L}_\sigma]$
has a filtration whose subquotients are isomorphic to 
cell modules. Moreover,
the multiplicity of  $\Bbbk[\mathbf{L}_\tau]$ in
$\Bbbk[\mathbf{L}_\rho]\otimes_\Bbbk\Bbbk[\mathbf{L}_\sigma]$
equals
\begin{equation}\label{eq5}
\mathbf{p}_{k,l}^m=\sum_{i=0}^k\sum_{j=0}^l
(-1)^{k+l-i-j}\binom{k}{i}\binom{l}{j}\binom{ij}{m}.
\end{equation}
\end{proposition}

\begin{proof}
Similarly to Subsection~\ref{s4.4}, existence of the filtration follows
from the fact that the action of $S_\tau$ on the elements of the form
$\xi\otimes \zeta$ described above is free.

To compute the multiplicity, for a fixed $\tau$, we need to count 
the number of the elements of the form  $\xi\otimes \zeta$ as above, with
$\mathrm{cod}(\xi)=\tilde{\rho}$ and
$\mathrm{cod}(\zeta)=\tilde{\sigma}$, for which
$\tilde{\rho}\cap \tilde{\sigma}=\tau$.

Given $M\in Q_{k,l}^m$, there are $m!$ bijections between the 
parts of $\tau$ and the ``$1$''-entries in $M$. If we fix such a bijection $\psi$,
we can read off the parts of  $\tilde{\rho}$ by uniting the parts of 
$\tau$ along the rows of $M$. Also, we can read off the parts of 
$\tilde{\sigma}$ by uniting the parts of $\tau$ along the columns of $M$.
This determines the codomains of $\xi$ and $\zeta$, respectively
(note that there is no fixed order on the parts of these codomains,
which explains the equivalence with respect to the action of $G$).
As $\rho$ and $\sigma$ are fixed, to determine $\xi$ and $\zeta$,
it remains to choose $\overline{\xi}$, in $k!$ different ways,
and $\overline{\zeta}$, in $l!$ different ways. 

Now let $\mathcal{Y}(M,\psi)$ be the set of all $\xi\otimes \zeta$ described above. 
Let $\mathcal{Y}(M)$ be the union of $\mathcal{Y}(M,\psi)$ over all bijections 
$\psi$ between the parts of $\tau$ and the ``$1$''-entries of $M$. 
If $M$ and $N$ belong to the same $G$-orbit $\mathcal{O}$, then 
$\mathcal{Y}(M)=\mathcal{Y}(N)$, we denote this set by $\mathcal{Y}_\mathcal{O}$. 
The cardinality of $\mathcal{Y}_\mathcal{O}$ is 
$\frac{m!\cdot k!\cdot l! }{|\text{stab}_\mathcal{O}|}$. 
Let $\mathcal{Y}$ be union of $\mathcal{Y}_\mathcal{O}$ over the orbits 
$\mathcal{O}$ in $Q^{m}_{k,l}/G$. Note that the set $\mathcal{Y}$ consists 
of $\xi\otimes\zeta$ with $\mathrm{cod}(\xi)\cap \mathrm{cod}(\zeta)=\tau$ 
with $\xi\in\mathbf{L}_\rho$ and $\zeta\in\mathbf{L}_\sigma$, and it has the cardinality:
\begin{equation*}
    m!\cdot\sum_{\mathcal{O}\in Q_{k,l}^m/G} \frac{k!\cdot l!}{|\mathrm{stab}_\mathcal{O}|}.
\end{equation*}
Since the action of $S_\tau$ on $\mathcal{Y}$ is free, dividing
by $|S_\tau|=m!$, gives 
\begin{displaymath}
\sum_{\mathcal{O}\in Q_{k,l}^m/G} \frac{k!\cdot l!}{|\mathrm{stab}_\mathcal{O}|}
=\mathbf{p}_{k,l}^m.
\end{displaymath}
The right hand side of Formula~\eqref{eq5} is obtained using inclusion-exclusion principle
with respect to the requirement that all rows and all columns of elements
in $Q_{k,l}^m$ should be non-zero.
\end{proof}

It would be interesting to have 
a combinatorial formula for $\mathbf{q}_{k,l}^m$. Although it is easy
to determine $\mathbf{q}_{k,l}^m$ in some special cases
(for instance, $\mathbf{q}_{k,l}^{kl}=1$ and 
$\mathbf{q}_{k,l}^{\max\{k,l\}}$ is the number of partitions of
$\max\{k,l\}$ with exactly $\min\{k,l\}$ parts), we do not
know the answer in general.

\begin{example}\label{ex6}
{\em 
Consider the case $k=l=m=n$. In this case, 
$Q_{n,n}^n$ is just the set of all permutation matrices and
hence $Q_{n,n}^n/G$ is a singleton (as any permutation matrix
can be reduced to the identity matrix by permutation of 
rows or columns). This means that we have $\mathbf{q}_{n,n}^n=1$, 
and, moreover,  $\mathbf{p}_{n,n}^n=n!$ which is exactly the multiplicity 
of $\Bbbk[S_n]$ in $\Bbbk[S_n]\otimes_\Bbbk\Bbbk[S_n]$.
}
\end{example}

In general it is easy to see that if $\mathbf{q}^m_{k,l}=1,$ 
then the multiplicity of $\Bbbk[\mathbf{L}_\tau]$ in 
$\Bbbk[\mathbf{L}_\rho]\otimes \Bbbk[\mathbf{L}_\sigma]$ 
is the cardinality of the set $Q^m_{k,l}$.

\subsection{Tensor product of simple modules}\label{s5.4}

The results of Subsection~\ref{s5.3} suggest what should be done
in order to understand tensor products of simple $I_n^*$-modules.
This seems to be a rather difficult problem in the general case.
Below we outline how to apply Subsection~\ref{s5.3} to this problem in more detail.

For $0\leq k,l,m\leq n$, consider the corresponding set $Q_{k,l}^m$.
Fix:
\begin{itemize}
\item a set partition $\rho$ of $\underline{n}$ with 
exactly $k$ parts, 
\item a set partition $\sigma$ of $\underline{n}$ with 
exactly $l$ parts, 
\item a set partition $\tau$ of $\underline{n}$ with 
exactly $m$ parts. 
\end{itemize}
%
Recall the set $\mathcal{Y}$ from the proof of Proposition~5. 
Note that, for $\xi\otimes \zeta\in \mathcal{Y}$, we 
have $\varepsilon_{\tau} (\xi\otimes \zeta)=\xi\otimes \zeta$
and, in fact, $\varepsilon_\tau$ is the maximum 
(w.r.t. $\preceq$) element in the set of all
idempotents in $I_n^*$ that do not annihilate $\xi\otimes \zeta$.
Also, note that $\Bbbk \,\mathcal{Y}$ has the natural structure of a 
$\Bbbk[S_m]$-$\Bbbk[S_k\times S_l]$-bimodule,
where the group $S_m$ acts on the left by permuting the parts of $\tau$, 
the group $S_k$ acts on the right by permuting the parts of $\rho$
and the group $S_l$ acts on the right by permuting the parts of $\sigma$.

Similarly, to \eqref{eq-s4.9-1}, 
for $\lambda\vdash k$, $\mu\vdash l$ and $\nu\vdash m$, we have
\begin{displaymath}
[\mathbf{N}^\lambda\otimes_\Bbbk\mathbf{N}^\mu:\mathbf{N}^\nu]=
[\Bbbk \,\mathcal{Y}\otimes_{\Bbbk[S_k\times S_l]}
(\mathbf{S}^\lambda\otimes_\Bbbk\mathbf{S}^\mu):\mathbf{S}^\nu].
\end{displaymath}
By construction, the bimodule $\Bbbk \,\mathcal{Y}$ splits into 
a direct sum of subbimodules given by varying $M$ from the previous
paragraph inside a fixed $S_k\times S_l$-orbit in $Q_{k,l}^m$.
For a fixed such orbit $\mathcal{O}$, we denote by
$\Bbbk\mathcal{Y}_{\mathcal{O}}$ the corresponding direct summand
of $\Bbbk \,\mathcal{Y}$. Therefore, the real problem is to 
determine the multiplicities
\begin{displaymath}
[\Bbbk \,\mathcal{Y}_{\mathcal{O}}\otimes_{\Bbbk[S_k\times S_l]}
(\mathbf{S}^\lambda\otimes_\Bbbk\mathbf{S}^\mu):\mathbf{S}^\nu],
\end{displaymath}
for each $\mathcal{O}\in Q_{k,l}^m/_{S_k\times S_l}$.
Below we give a more group-theoretic reformulation of this problem
and discuss various examples of different orbits. An explicit solution
in all cases seems to be quite difficult.

In the above construction, assume that 
$\mathcal{O}=(S_k\times S_l)\cdot M$ and 
let $H$ be the stabilizer of $M$ in $S_k\times S_l$. Consider the set 
$A$ of all possible bijections between
the elements of $\underline{m}$ and the ``$1$''-entries in $M$.
Then each element of $H$ gives rise to a permutation of 
the ``$1$''-entries in $M$, defining a homomorphism from $H$
to $S_m$. Moreover, the permutation in $S_m$ corresponding to a
non-trivial element of $H$ is, obviously, non-trivial. Therefore
this homomorphism is injective and thus we can identify
$H$ as a subgroup of $S_m$.

\begin{proposition}\label{prop7}
We have
\begin{displaymath}
\left[\Bbbk \,\mathcal{Y}_{\mathcal{O}}\otimes_{\Bbbk[S_k\times S_l]}
\big(\mathbf{S}^\lambda\otimes_\Bbbk\mathbf{S}^\mu\big):
\mathbf{S}^\nu\right]=
\left[\mathrm{Ind}^{S_m}_H\mathrm{Res}^{S_k\times S_l}_H
\big(\mathbf{S}^\lambda\otimes_\Bbbk\mathbf{S}^\mu\big)
:\mathbf{S}^\nu\right]. 
\end{displaymath}
\end{proposition}

\begin{proof}
For $G=S_k\times S_l$,
the $\Bbbk[S_m]$-$\Bbbk[G]$-bimodule corresponding to the functor
$\mathrm{Ind}^{S_m}_H\mathrm{Res}^G_H$ is given by
\begin{displaymath}
B:=\Bbbk[S_m]\otimes_{\Bbbk[H]}\Bbbk[G].
\end{displaymath}
Note that the dimension of this bimodule is $\frac{|S_m||G|}{|H|}$
and that this coincides with the dimension of $\Bbbk \,\mathcal{Y}_{\mathcal{O}}$.
From our definition of $H$, it follows that the
right action of $H$ on $\mathcal{Y}_{\mathcal{O}}$ (i.e. the action in terms of $G$) 
coincides with the left action of $H$ on $\mathcal{Y}_{\mathcal{O}}$ (i.e. the 
action in terms of $S_m$).

Therefore, sending $e\otimes e\in B$ to 
the element in $\mathcal{Y}_{\mathcal{O}}$ corresponding to $(M,\psi)$, 
gives rise to a homomorphism of $\Bbbk[S_m]$-$\Bbbk[G]$-bimodules from $B$ to 
$\Bbbk \,\mathcal{Y}_{\mathcal{O}}$.
Since the action of $S_m\times G$ on $\mathcal{Y}_\mathcal{O}$ is transitive, this
homomorphism is surjective. By comparing the dimensions, we
obtain that this homomorphism is an isomorphism.
\end{proof}

As an immediate corollary of Proposition~\ref{prop5}
and Proposition~\ref{prop7}, we have:

\begin{corollary}\label{cor-prop7}
The multiplicity of $\mathbf{N}^{\nu}$ in 
$\mathbf{N}^{\lambda}\otimes_\Bbbk\mathbf{N}^{\mu}$ 
is given by the sum of the multiplicities in Proposition~\ref{prop7},
taken along some cross-section of $Q_{k,l}^m/_{S_k\times S_l}$.
\end{corollary}

\begin{example}\label{ex8}
{\em 
Consider the extreme case $m=kl$. In this case, $G=H$ is a subgroup of
$S_{m}$ and hence we are interested in the multiplicities 
\begin{displaymath}
\left[\mathrm{Ind}^{S_m}_{S_k\times S_l}
\big(\mathbf{S}^\lambda\otimes_\Bbbk\mathbf{S}^\mu\big):
\mathbf{S}^\nu\right]. 
\end{displaymath}
Note that these are not the Littlewood--Richardson coefficient
as our $S_k\times S_l$ is not a Young subgroup of $S_m$.
This example appears in \cite{Ryba} and will be studied in more 
detail in Section~\ref{s7}.
The extreme example of this case is $k=1$, when we get $H=G=S_m$.
}
\end{example}

\begin{example}\label{ex9}
{\em 
Consider the other extreme case $m=l>k$. Then we can assume that 
the $i$-th row of $M$ has exactly $x_i$ entries equal to $1$
and $x_1\geq x_2\geq \dots \geq x_k$ with $x_1+x_2+\dots +x_k=l$.
In other words, we have a partition $\mathbf{x}$ of $l$
with exactly $k$ parts. Note that each column of $M$
has exactly one entry equal to $1$,  since $m=l$.
Assume that exactly $y_1$ parts of
this partition are equal to $z_1=x_1$, then exactly $y_2$
parts are equal to the next size $z_2$ of the parts in $\mathbf{x}$
and so on. Then the group $H$ is isomorphic to the direct product
$(S_{y_1}\wr S_{z_1})\times(S_{y_2}\wr S_{z_2})\times\cdots$ of
wreath products of symmetric groups. The extreme example of this
case is again $k=1$, when we get $H=G=S_m$.}
%

\end{example}

\begin{example}\label{ex9n}
{\em 
Consider the case $m=(l-1)+(k-1)$ and take $M$ such that 
the first row of $M$ is $(0,1,1,\dots,1)$ and the first 
column of $M$  is $(0,1,1,\dots,1)^t$. Then the remaining
entries of $M$ are zero. In this case
$H\cong S_{k-1}\times S_{l-1}$ is both, a Young subgroup of $G$,
and a Young subgroup of $S_m$. Let us denote by $\to$
the usual branching relation on partitions (i.e. $\xi\to\eta$
means that $\xi$ is obtained from $\eta$ by removing a removable node
in the Young diagram).
Then the multiplicities described in Proposition~\ref{prop7}
are given by the formula
\begin{displaymath}
\sum_{\alpha\to \lambda}\sum_{\beta\to \mu}\mathbf{c}_{\alpha,\beta}^\nu. 
\end{displaymath}
}
\end{example}

\begin{example}\label{ex9nn}
{\em 
Consider the case $l=k$, $m=\frac{l(l+1)}{2}$ and take $M$ 
such that it  is upper triangular. In this case
$H$ is the trivial group (a singleton). Then the multiplicities 
described in Proposition~\ref{prop7} are given by the formula
$\dim(\mathbf{S}^\lambda)\cdot\dim(\mathbf{S}^\mu)\cdot\dim(\mathbf{S}^\nu)$.
}
\end{example}

\section{The inverse 
semigroup of all bijections between subquotients 
of a finite set}\label{s6}

\subsection{The semigroup  $PI_n^*$}\label{s6.1}

Consider the semigroup $PI_n^*$ defined in  \cite[Subsection~2.2]{KM}. 
The elements of  $PI_n^*$ are all possible bijections
between quotients of subsets of $\underline{n}$ (equivalently, 
between subsets of quotients of $\underline{n}$) and the 
operation is the natural version of composition for such bijections.
Here is an example for $n=8$, where composition is from right to left and
the elements are presented using the two-row notation for functions:

\resizebox{\textwidth}{!}{
$
\left(\begin{array}{cccc}\{1\}&\{3,5\}&\{6,7\}&\{8\}\\
\{1,2\}&\{3,4\}&\{5\}&\{6,8\}
\end{array}\right) \circ
\left(\begin{array}{cccc}\{2,3\}&\{4\}&\{5,6\}&\{7,8\}\\
\{1,2\}&\{3\}&\{5\}&\{6,7,8\}\end{array}\right) =
\left(\begin{array}{cc}\{4,5,6\}&\{7,8\}\\\{3,4\}&\{5,6,8\}\end{array}\right) 
$
}

The main point in this multiplication is that, as soon some element $x$ in the
product hits an undefined part of the other element, all elements connected
to $x$ become undefined. In the above example, this happens for $x=2$ in the
right factor. We do refer the reader to \cite[Subsection~2.2]{KM} for the
full definition which is very technical. The semigroup $PI_n^*$ is, in some 
sense, a mixture between  $IS_n$ and $I_n^*$.

Each element $\xi$ in $PI_n^*$ is thus uniquely given by the following
data:
\begin{itemize}
\item a set partition $\xi_d$ of some subset $X$ of $\underline{n}$, called 
the {\em domain} of $\xi$;
\item a set partition $\xi_r$ of some subset $Y$ of $\underline{n}$
(with the same number of parts as $\xi_d$), called 
the {\em codomain} of $\xi$;
\item a bijection $\overline{\xi}$ from $X/\xi_d$ to $Y/\xi_r$.
\end{itemize}
The cardinality of $X/\xi_d$ is called the 
{\em rank} of $\xi$ and denoted $\mathrm{rank}(\xi)$.
We also call the same number the {\em rank} of $\xi_d$.

Let $\xi,\zeta\in PI_n^*$. Recall, see for example 
\cite[Proposition~3]{KM}, the following description of Green's relations
$\mathcal{L}$, $\mathcal{R}$, $\mathcal{H}$ and $\mathcal{D}=\mathcal{J}$
for $PI_n^*$:
\begin{itemize}
\item $\xi\mathcal{L}\zeta$ if and only if $\mathrm{dom}(\xi)=\mathrm{dom}(\zeta)$,
\item $\xi\mathcal{R}\zeta$ if and only if $\mathrm{cod}(\xi)=\mathrm{cod}(\zeta)$,
\item $\xi\mathcal{H}\zeta$ if and only if $\mathrm{dom}(\xi)=\mathrm{dom}(\zeta)$
and $\mathrm{cod}(\xi)=\mathrm{cod}(\zeta)$,
\item $\xi\mathcal{J}\zeta$ if and only if $\mathrm{rank}(\xi)=\mathrm{rank}(\zeta)$.
\end{itemize}

The idempotents of $PI_n^*$ are naturally identified with equivalence
relations on subsets of $\underline{n}$: given an equivalence relation
$\rho$ on some subset $X$ of $\underline{n}$, the corresponding 
idempotent $\varepsilon_\rho$ is the identity map on  $X/\rho$. 
The maximal subgroup of $PI_n^*$ corresponding to 
$\rho$ is the symmetric group $S_\rho$ on the equivalence classes of $\rho$.

The semigroup $PI_k^*$ appears in the Schur-Weyl  duality for the $k$-th tensor 
power of the direct sum of the natural and the trivial representations 
of $IS_n$, see \cite[Theorem~2]{KMaz}.

\subsection{Cells and simple $PI_n^*$-modules}\label{s6.2}

Let $\rho$ be an equivalence relation on some subset $X$ of  $\underline{n}$.
Let $\mathbf{L}_\rho$ be a left cell (i.e. an $\mathcal{L}$-equivalence class) 
in  $PI_n^*$ containing $\varepsilon_\rho$.

The linearization $\Bbbk[\mathbf{L}_\rho]$ is naturally a left $PI_n^*$-module 
with the  action given as follows, for $\xi\in PI_n^*$ and 
$\zeta\in \mathbf{L}_\rho$:
\begin{displaymath}
\xi\cdot\zeta=
\begin{cases}
\xi\zeta,& \text{ if } \xi\zeta\in \mathbf{L}_\rho;\\
0,& \text{ otherwise.}
\end{cases}
\end{displaymath}
This is the cell $PI_n^*$-module corresponding to $\mathbf{L}_\rho$.

The maximal subgroup $S_\rho$ acts on $\mathbf{L}_\rho$ by multiplication 
on the right. This equips $\Bbbk[\mathbf{L}_\rho]$ with the natural 
structure of a $\Bbbk[PI_n^*]$-$\Bbbk[S_\rho]$-bimodule.

For any $\lambda\vdash \mathrm{rank}(\varepsilon_\rho)$, 
the $\Bbbk[PI_n^*]$-module
\begin{displaymath}
\mathbf{N}^{\lambda}:=\Bbbk[\mathbf{L}_\rho]\bigotimes_{\Bbbk[S_\rho]}
\mathbf{S}^\lambda
\end{displaymath}
is simple. Up to isomorphism, it does not depend on the choice of 
$\mathbf{L}_\rho$ inside its $\mathcal{J}$-cell (i.e. for the fixed rank of 
$\varepsilon_\rho$). Moreover, up to isomorphism,
each simple $\Bbbk[PI_n^*]$-module has such a form,  for some 
$\rho$ and $\lambda$. We refer to \cite{St} for details.

\subsection{Tensor product of cell  modules}\label{s6.3}

Let $\rho$ and $\sigma$ be two set partitions 
(a.k.a. equivalence relations) of two subsets $X$ and $Y$ of $\underline{n}$, respectively.
The tensor product
\begin{displaymath}
\Bbbk[\mathbf{L}_\rho]\otimes_\Bbbk\Bbbk[\mathbf{L}_\sigma]
\end{displaymath}
is an $PI_n^*$-module via the diagonal action of $PI_n^*$. 
In this subsection we are going to show that 
this module has a filtration whose subquotients are isomorphic to cell
modules. Also, we are going to describe a combinatorial object which
determines the corresponding multiplicities (just like in the cases
of $IS_n$ and $I_n^*$, the fact that these multiplicities are well-defined
follows directly from the semi-simplicity of the semigroup algebra
of $PI_n^*$).

Set $k=\mathrm{rank}(\varepsilon_\rho)$, $l=\mathrm{rank}(\varepsilon_\sigma)$.
Let $U$ be a subset of $\underline{n}$ of cardinality at least $k$,
and $\tilde{\rho}$ a set partiton of $U$ of the same rank as $\rho$. 
Let $V$ be a subset of $\underline{n}$ of cardinality at least $l$,
and $\tilde{\sigma}$ a set partiton of $V$ of the same rank as $\sigma$. 
Consider the set partition $\tau$ of $U\cup V$ whose parts are
exactly the following sets:
\begin{itemize}
\item all non-empty  intersections of a part of $\tilde{\rho}$ with 
a part of $\tilde{\sigma}$;
\item all non-empty  intersections of a part of $\tilde{\rho}$ with 
$\underline{n}\setminus V$;
\item all non-empty  intersections of a part of $\tilde{\sigma}$ with 
$\underline{n}\setminus U$.
\end{itemize}

Let $\xi\in\mathbf{L}_\rho$ and $\zeta\in\mathbf{L}_\sigma$
be such that $\mathrm{cod}(\xi)=\tilde{\rho}$ and 
$\mathrm{cod}(\zeta)=\tilde{\sigma}$. Then $\tau$ is the maximum
element in the set of all set partitions $\omega$ of $U\cup V$
such that $\varepsilon_\omega\xi\in \mathbf{L}_\rho$
and $\varepsilon_\omega\zeta\in\mathbf{L}_\sigma$.
Note also that the action of $S_\tau$ on $\xi\otimes \zeta$ is, clearly, free.
All this means exactly that the element
$\xi\otimes \zeta$ contributes to a cell subquotient
$\Bbbk[\mathbf{L}_\tau]$ of 
$\Bbbk[\mathbf{L}_\rho]\otimes_\Bbbk\Bbbk[\mathbf{L}_\sigma]$.

To provide a combinatorial object for counting the multiplicities,
we need to invert our problem. Given some $\tau$, we need to find
the number of all possible $\tilde{\rho}$ and $\tilde{\sigma}$
as above which are connected to $\tau$ also in the above sense.

Set $m:=\mathrm{rank}(\varepsilon_\tau)$.
Let $\tilde{Q}_{k,l}^{m}$ be the set of all $0/1$-matrices 
whose rows are indexed by the elements in 
$\{0,1,2,\dots,k\}$ and whose columns are 
indexes by the elements in  $\{0,1,2,\dots,l\}$
satisfying the following conditions:
\begin{itemize}
\item the $(0,0)$-entry is $0$;
\item exactly $m$ entries are equal to $1$;
\item all rows with positive indices are non-zero;
\item all columns with positive indices are non-zero.
\end{itemize}
The group $G:=S_k\times S_l$ acts on $\tilde{Q}_{k,l}^m$
by (independent) permutation of rows and columns
with positive indices. Consider
the corresponding orbit set $\tilde{Q}_{k,l}^m/G$ and set
$\tilde{\mathbf{p}}_{k,l}^m:=|\tilde{Q}_{k,l}^m|$ and
$\tilde{\mathbf{q}}_{k,l}^m:=|\tilde{Q}_{k,l}^m/G|$. Note that,
for $\tilde{\mathbf{q}}_{k,l}^m$ to be non-zero, we must 
have the inequality $\max\{k,l\}\leq m\leq (k+1)(l+1)-1$.
It would be interesting to have 
a combinatorial formula for $\tilde{\mathbf{q}}_{k,l}^m$.

\begin{proposition}\label{prop6-1}
The module $\Bbbk[\mathbf{L}_\rho]\otimes_\Bbbk\Bbbk[\mathbf{L}_\sigma]$ has
a filtration whose subquotients are isomorphic to cell modules.
Moreover, the multiplicity of $\Bbbk[\mathbf{L}_\tau]$ in
$\Bbbk[\mathbf{L}_\rho]\otimes_\Bbbk\Bbbk[\mathbf{L}_\sigma]$
equals 
\begin{equation}\label{eq5n}
\tilde{\mathbf{p}}_{k,l}^m=
\sum_{i=0}^k\sum_{j=0}^l
(-1)^{k+l-i-j}\binom{k}{i}\binom{l}{j}\binom{(i+1)(j+1)-1}{m}
\end{equation}
\end{proposition}

\begin{proof}
Mutatis mutandis the proof of Proposition~\ref{prop5}.
\end{proof}

\begin{example}\label{ex6-2}
{\em 
Consider the case $k=m=l=n$. In this case, 
$\tilde{Q}_{n,n}^n$ is just the set of all permutation matrices 
on positive indices (and both the $0$-th row and the $0$-th column are zero).
Hence $\tilde{Q}_{n,n}^n/G$ is a singleton 
and thus  $\tilde{\mathbf{q}}_{n,n}^n=1$. 
We also have $\tilde{\mathbf{p}}_{n,n}^n=n!$,
which is exactly the multiplicity 
of $\Bbbk[S_n]$ in $\Bbbk[S_n]\otimes_\Bbbk\Bbbk[S_n]$.
}
\end{example}

\subsection{Tensor product of simple modules}\label{s6.4}

Similarly to the case of the dual symmetric inverse monoid,
the results of Subsection~\ref{s6.3} suggest what should be done
in order to understand tensor products of simple $PI_n^*$-modules.
This seems to be a rather difficult problem. Below we outline
how to apply Subsection~\ref{s6.3} to this problem in more detail.

Let $\rho$ be a set partition of $X\subset\underline{n}$ 
with $k$ parts and $\sigma$ a set partition of
$Y\subset\underline{n}$ with $l$ parts. Let $\tau$ be
the set partition of $X\cup Y$ given by intersecting 
the parts of $\rho$ with the parts of $\sigma$ and,
additionally, the parts of $\rho$ with $\underline{n}\setminus Y$
and the parts of $\sigma$ with $\underline{n}\setminus X$. 
Then we have $\varepsilon_{\tau} (\varepsilon_{\rho}\otimes 
\varepsilon_{\sigma})=\varepsilon_{\rho}\otimes \varepsilon_{\sigma}$
and, in fact, $\varepsilon_\tau$ is the maximum 
(w.r.t. $\preceq$) element in the set of all
idempotents in $PI_n^*$ that do not annihilate 
$\varepsilon_{\rho}\otimes \varepsilon_{\sigma}$.
Assume that $\tau$ has exactly $m$ parts.

Fix some linear 
orderings for the parts of $\rho$ and for the parts of $\sigma$.
Then we have the matrix $M\in \tilde{Q}_{k,l}^m$ defined as follows:
\begin{itemize}
\item For positive $i$ and $j$, the $(i,j)$-th entry of 
$M$ is $1$ if and only if the intersection of the $i$-th part of 
$\rho$ with the $j$-th part of $\sigma$ is not empty;
\item The $(0,j)$-th entry of $M$ is $1$ if and only if the 
intersection of the $j$-th part of $\sigma$ with 
$\underline{n}\setminus X$ is not empty;
\item The $(i,0)$-th entry of $M$ is $1$ if and only if the 
intersection of the $i$-th part of $\rho$ with 
$\underline{n}\setminus Y$ is not empty.
\end{itemize}

Consider the $G$-orbit $G\cdot M$ of $M$ inside $\tilde{Q}_{k,l}^m$
and denote by $\mathbf{M}$ the set of all possible bijective 
decorations of the ``$1$''-entries of the elements in
$G\cdot M$ by the elements of $\underline{m}$. Clearly, 
both $S_m$ and $G$ act on $\mathbf{M}$ and these two actions
commute. Hence the linearization $\Bbbk[\mathbf{M}]$ of 
$\mathbf{M}$ becomes an $S_m$-$G$-bimodule.

Let $\lambda\vdash k$, $\mu\vdash l$ and $\nu\vdash m$.
In order to understand tensor products of simple $PI_n^*$-modules,
Proposition~\ref{prop6-1} (more precisely,  it's proof based on the
proof of Proposition~\ref{prop5}) implies that 
we need to understand the multiplicity of $\mathbf{S}^\nu$ in the
module
\begin{displaymath}
\Bbbk[\mathbf{M}]\otimes_{\Bbbk[G]}
\big(\mathbf{S}^\lambda\otimes_\Bbbk\mathbf{S}^\mu\big).
\end{displaymath}

Let $H$ be the stabilizer of $M$ in $G$. Consider the set 
$A$ of all possible bijections between
the elements of $\underline{m}$ and the ``$1$''-entries in $M$.
Then each element of $H$ gives rise to a permutation of 
the ``$1$''-entries in $M$, defining a homomorphism from $H$
to $S_m$. Moreover, the permutation in $S_m$ corresponding to a
non-trivial element of $H$ is, obviously, non-trivial. Therefore
this homomorphism is injective and thus we can canonically identify
$H$ as a subgroup of $S_m$.

\begin{proposition}\label{prop6.3n}
We have
\begin{displaymath}
\left[\Bbbk[\mathbf{M}]\otimes_{\Bbbk[G]}
\big(\mathbf{S}^\lambda\otimes_\Bbbk\mathbf{S}^\mu\big):
\mathbf{S}^\nu\right]=
\left[\mathrm{Ind}^{S_m}_H\mathrm{Res}^G_H
\big(\mathbf{S}^\lambda\otimes_\Bbbk\mathbf{S}^\mu\big)
:\mathbf{S}^\nu\right]. 
\end{displaymath}
\end{proposition}

\begin{proof}
Mutatis mutandis the proof of Proposition~\ref{prop7}.
\end{proof}

\begin{corollary}\label{cor-prop6.3n}
The multiplicity of $\mathbf{N}^{\nu}$ in 
$\mathbf{N}^{\lambda}\otimes_\Bbbk\mathbf{N}^{\mu}$ 
is given by the sum of the multiplicities in Proposition~\ref{prop6.3n},
taken along some cross-section of $\tilde{Q}_{k,l}^m/G$.
\end{corollary}

\begin{example}\label{ex6-4n}
{\em 
Consider the case where $m=k+l$ and $M$ is the matrix in which all ``1''-entries are
either in the $0$-th row or in the $0$-th column. In this case, $G=H$ is a 
Young subgroup of $S_{m}$ and hence the multiplicities in 
Proposition~\ref{prop6.3n} we are interested in are exactly 
the Littlewood--Richardson coefficients $\mathbf{c}_{\lambda,\mu}^\nu$.
}
\end{example}

\begin{example}\label{ex6-4nn}
{\em 
Consider the case $k=l$ and $m=3k$, where $M$ is the matrix in which 
all ``1''-entries are either in the $0$-row or in the $0$-column or on
the main diagonal. In this case, $H$ is isomorphic to $S_k$ and is diagonally
embedded both into $G$ and into the Young subgroup of $S_{m}$ isomorphic to 
$S_k\times S_k\times S_k$. Since these embeddings are diagonal, the
restriction from $G$ to $H$ and the
induction from $H$ to $S_k\times S_k\times S_k$ are computed using
Kronecker coefficients. After that, the induction from 
$S_k\times S_k\times S_k$ to $S_m$ is computed using 
Littlewood--Richardson coefficients. Consequently, the multiplicities in 
Proposition~\ref{prop6.3n} are given by the following expression:
\begin{displaymath}
\sum_{\alpha\vdash k}\sum_{\beta\vdash k}
\sum_{\gamma\vdash k}\sum_{\delta\vdash k}
\sum_{\epsilon\vdash k}\sum_{\kappa\vdash 2k}
\mathbf{c}_{\epsilon,\kappa}^\nu
\mathbf{c}_{\gamma,\delta}^\kappa
\mathbf{g}_{\delta,\epsilon}^\beta
\mathbf{g}_{\beta,\gamma}^\alpha
\mathbf{g}_{\lambda,\mu}^\alpha.
\end{displaymath}
}
\end{example}

\section{Branching related to $S_k\times S_l$ as a subgroup of 
$S_{kl}$}\label{s7}

\subsection{Setup}\label{s7.1}

In this section we look at Example~\ref{ex8} in more detail.
Let $m=kl$ and let $M$ be the $k\times l$ matrix in which all
entries are equal to $1$. We fix the bijection between $\underline{m}$
and the entries of this matrix which reads the entries left-to-right
and then top-to-bottom. The group $S_k$ acts by permuting the rows of 
$M$, which identifies it as a subgroup of $S_m$. The group $S_l$ acts by 
permuting the columns of $M$, which identifies it as a subgroup of $S_m$. 
Put together, we have identified $S_k\times S_l$ as a subgroup of 
$S_{kl}$.

For $\lambda\vdash k$, $\mu\vdash l$ and $\nu\vdash kl$, denote by
$\mathbf{b}^\nu_{\lambda,\mu}$ the multiplicity of 
$\mathbf{S}^\lambda\otimes_\Bbbk\mathbf{S}^\mu$ in the restriction of
$\mathbf{S}^\nu$ from $S_{kl}$ to $S_k\times S_l$.

\subsection{Partition algebra}\label{s7.2}

For $n\geq 0$ and $x\in\Bbbk\setminus\{0\}$ consider the corresponding 
partition algebra $\mathcal{P}_n(x)$, as defined in \cite{Jo,Mar}.
It is a $\Bbbk$-algebra that has a basis consisting of 
set partitions of two (the upper and the lower) copies of $\underline{n}$. Such set 
partitions are usually represented by certain diagrams
(partition diagrams) whose connected components, viewed as graphs, represent
the parts of the partition, for example:
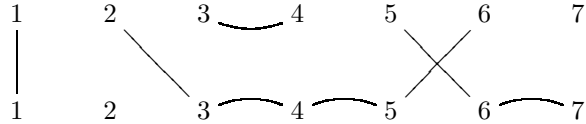
\begin{displaymath}
\xymatrix{1\ar@{-}[d]&2\ar@{-}[dr]&3\ar@/_5pt/@{-}[r]&4&
5\ar@{-}[dr]&6\ar@{-}[dl]&7\\
1&2&3\ar@/^5pt/@{-}[r]&4\ar@/^5pt/@{-}[r]&5&6\ar@/^5pt/@{-}[r]&7
}
\end{displaymath}
Note that such a diagram representing a partition is not unique.
Algebra multiplication is defined in 
terms of concatenation of such diagrams followed by a certain
straightening procedure which results in a new diagram and
a non-negative integer  $r$ (the number of parts removed during the
straightening procedure). The product is then this new diagram 
times $x^r$. Here is an example:

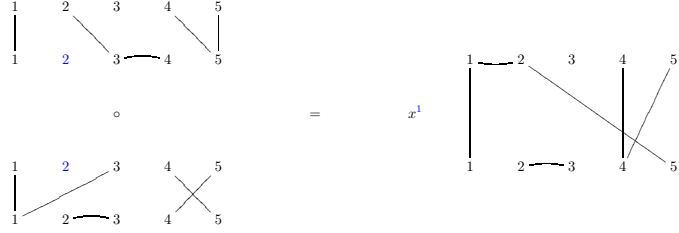
\begin{center}
\resizebox{9cm}{!}{
$
\xymatrix{
1\ar@{-}[d]&2\ar@{-}[dr]&3&4\ar@{-}[dr]&5\ar@{-}[d]&&&&&&&&&\\
1&{\color{blue}2}&3\ar@/^3pt/@{-}[r]&4&5&&&&&
1\ar@{-}[dd]\ar@/_3pt/@{-}[r]&2\ar@{-}[ddrrr]&3&4\ar@{-}[dd]&5\ar@{-}[ddl]\\
&&\circ&&&&=&&x^{\color{blue}1}&&&&&\\
1\ar@{-}[d]&{\color{blue}2}&3\ar@{-}[dll]&4\ar@{-}[dr]&5\ar@{-}[dl]&&&&&
1&2\ar@/^3pt/@{-}[r]&3&4&5\\
1&2\ar@/^3pt/@{-}[r]&3&4&5&&&&&&&&&\\
}
$
}
\end{center}

The unique removed part in this example is colored in {\color{blue}blue}.
We refer the reader to \cite{HR} for details.

Given a partition diagram $\mathbf{d}$, the parts which 
intersect both copies of $\underline{n}$ are called 
{\em propagating lines} and the number of such parts 
is called {\em rank}.
This rank is an integer between $0$ and $n$. For
$0\leq i\leq n$, all partition diagrams of rank at most $i$
span a two-sided ideal in $\mathcal{P}_n(x)$, which we denote by
$I_i$. Given a simple $\mathcal{P}_n(x)$-module $L$, there
is a minimal $i$ such that $I_i$ does not belong to the
annihilator of $L$. Simple $\mathcal{P}_n(x)$-modules with such 
property are in a natural bijection with partitions of $i$
and are constructed as follows.

Let $\mathbf{d}$ be a partition diagram of rank $i$. Consider the set
$\mathbf{L}_{\mathbf{d}}$ of all partition diagrams of rank $i$
which can be obtained from $\mathbf{d}$ by left concatenation with other diagrams.
Then the linearization $\Bbbk\mathbf{L}_{\mathbf{d}}$ is naturally
a $\mathcal{P}_n(x)$-$\Bbbk[S_i]$-bimodule
which is free as a $\Bbbk[S_i]$-module. For $\lambda\vdash i$, the module
\begin{displaymath}
\mathbf{C}^\lambda:=\Bbbk\mathbf{L}_{\mathbf{d}}
\bigotimes_{\Bbbk[S_i]}\mathbf{S}^\lambda
\end{displaymath}
has simple top which we denote $\mathbf{N}^\lambda$. The set
\begin{displaymath}
\{ \mathbf{N}^\lambda\,:\, \lambda\vdash i, \, 0\leq i\leq n
\}
\end{displaymath}
is a full and irredundant set of representatives of isomorphism classes
of simple $\mathcal{P}_n(x)$-modules.

We also note that $\mathbf{C}^\lambda\cong \mathbf{N}^\lambda$
provided that $x\not\in\{0,1,2,\dots,2n-2\}$ since in this case
the partition algebra is semi-simple, see \cite{MaSa} or \cite[Theorem~3.27]{HR}.
Note that, for $x\in\{0,1,2,\dots,2n-2\}$, the isomorphism 
$\mathbf{C}^\lambda\cong \mathbf{N}^\lambda$ might still hold 
for some $\lambda$. In particular, under the assumption 
$x\not\in\{0,1,2,\dots,2n\}$, the module $\mathbf{N}^\lambda$ has a basis that does not
depend on $x$ and in which the action of each partition diagram is
given by a matrix in which all coefficients are polynomials in $x$.

We will call the modules $\mathbf{C}^\lambda$ {\em standard}.
A {\em standard filtration} of some $\mathcal{P}_n(x)$-module
is a filtration whose subquotients are standard modules.
This terminology reflects the fact that the algebra
$\mathcal{P}_n(x)$ is quasi-hereditary (here our assumption $x\neq 0$ is important).

\subsection{Schur-Weyl dualities for partition algebras}\label{s7.25}

Consider the natural $S_k$-module $\Bbbk^k$ and its $n$-fold
tensor power $(\Bbbk^k)^{\otimes n}$. It is an $S_k$-module
via the diagonal action of $S_k$. The algebra $\mathcal{P}_n(k)$ 
acts naturally on $(\Bbbk^k)^{\otimes n}$, see, for example,
\cite[Formula~(3.2)]{HR}, this action 
commutes with the action of $S_k$ and, moreover, the two 
commuting actions generate each other's centralizers. 
This is known as the {\em Schur-Weyl duality} for partition algebras,
see \cite[Section~3]{HR}. As a $\Bbbk[S_k]$-$\mathcal{P}_n(k)$ bimodule,
the space $(\Bbbk^k)^{\otimes n}$ decomposes into a direct 
sum of simple bimodules indexed by those partitions
$\lambda=(\lambda_1,\lambda_2,\dots)$ of $k$, for which $k-\lambda_1\leq n$. The corresponding 
simple $\Bbbk[S_k]$-$\mathcal{P}_n(k)$ bimodule is then
$\mathbf{S}^\lambda\otimes_\Bbbk\mathbf{N}^{\lambda_{>1}}$,
where $\lambda_{>1}$ is the partition of $k-\lambda_1$
obtained from $\lambda$ by deleting the first part $\lambda_1$.

We note that neither the action of $\Bbbk[S_k]$ nor the action of 
$\mathcal{P}_n(k)$ on $(\Bbbk^k)^{\otimes n}$ is faithful,
in general.

Now let $m=kl$. Then we have the natural $S_k$-module
$\Bbbk^k$ and the natural $S_l$-module
$\Bbbk^l$. Their tensor product 
$\Bbbk^k\otimes_{\Bbbk}\Bbbk^l$
is isomorphic to the natural $S_m$-module
$\Bbbk^m$. For any $n\geq 0$, we also have the obvious
isomorphism
\begin{displaymath}
(\Bbbk^k\otimes_{\Bbbk}\Bbbk^l)^{\otimes n} \cong
(\Bbbk^k)^{\otimes n}\otimes_{\Bbbk}(\Bbbk^l)^{\otimes n}.
\end{displaymath}
By the Schur-Weyl duality, the $S_{kl}$-endomorphisms of the left hand side
are given by  the image of $\mathcal{P}_n(kl)$. Similarly, the $S_{k}\times S_l$-endomorphisms 
of the right hand side are given by the image of 
$\mathcal{P}_n(k)\otimes_{\Bbbk}\mathcal{P}_n(l)$. 
This suggests that our original problem of restriction from
$S_{kl}$ to $S_{k}\times S_l$ should be interpretable, via the Schur-Weyl duality,
in terms of some induction problem for partition algebras. This is what we will investigate
in the next subsections.

\subsection{Generalized comultiplication for partition algebras}\label{s7.3}

We do not know any interesting comultiplication for partition algebras. However,
the following proposition describes a generalized version of  such comultiplication.

\begin{proposition}\label{prop7.3.1}
Let $n\in\mathbb{Z}_{>0}$ and $a,b\in\Bbbk$. Then, mapping a 
partition diagram $\mathbf{d}$ to $\mathbf{d}\otimes\mathbf{d}$,
defines an algebra homomorphism from $\mathcal{P}_n(ab)$
to $\mathcal{P}_n(a)\otimes_{\Bbbk}\mathcal{P}_n(b)$.
\end{proposition}

\begin{proof}
Let $\mathbf{d}$ and $\mathbf{d}'$ be two partition diagrams.
To compute their product we have to compose them and then 
apply the straightening procedure during which $r$ parts will be
removed and which outputs a new partition diagram $\mathbf{d}''$. 
Neither $r$ nor $\mathbf{d}''$ depend on the parameter of the 
partition algebra. 

Hence, for the algebra $\mathcal{P}_n(ab)$, the product of 
$\mathbf{d}$ and $\mathbf{d}'$ equals $(ab)^r\mathbf{d}''$.
At the same time, for the algebra 
$\mathcal{P}_n(a)\otimes_{\Bbbk}\mathcal{P}_n(b)$,
the product of $\mathbf{d}\otimes\mathbf{d}$ and 
$\mathbf{d}'\otimes\mathbf{d}'$ equals the element
$\big((a^r)\mathbf{d}''\big)\otimes\big((b^r)\mathbf{d}'')
=(ab)^r  \mathbf{d}''\otimes\mathbf{d}''$.
The claim follows.
\end{proof}

\subsection{Partition algebra and $\mathbf{b}^\nu_{\lambda,\mu}$}\label{s7.4}

Proposition~\ref{prop7.3.1} implies that, given a
$\mathcal{P}_n(a)$-module $X$ and a $\mathcal{P}_n(b)$-module $Y$,
the tensor product $X\otimes_{\Bbbk} Y$ is, naturally, a
$\mathcal{P}_n(ab)$-module. 

\begin{proposition}\label{prop7.4.1}
Assume that the $\mathcal{P}_n(a)$-module $X$ 
has a standard filtration and that the 
$\mathcal{P}_n(b)$-module $Y$ has a standard filtration.
Then the $\mathcal{P}_n(ab)$-module $X\otimes_{\Bbbk} Y$
has a standard filtration as well.
\end{proposition}

\begin{proof}
By additivity, it is enough to prove the claim under the assumption
that both $X$ and $Y$ are cell modules.

Let $\mathbf{L}_1$ and $\mathbf{L}_2$ be two left cells
such that $X=\Bbbk[\mathbf{L}_1]$ and $X=\Bbbk[\mathbf{L}_2]$.
Then $X$ has the canonical basis consisting of all 
partition diagrams in $\mathbf{L}_1$ and 
$Y$ has the canonical basis consisting of all 
partition diagrams in $\mathbf{L}_2$. Therefore
$X\otimes_{\Bbbk} Y$ has the canonical basis consisting of
all elements of the form $\mathbf{d}_1\otimes\mathbf{d}_2$,
where $\mathbf{d}_i\in \mathbf{L}_i$, for $i=1,2$.

Consider the oriented graph $\Gamma$ whose vertices are 
all these elements in the basis of $X\otimes_{\Bbbk} Y$.
For two vertices $\mathbf{d}_1\otimes\mathbf{d}_2$
and $\mathbf{d}'_1\otimes\mathbf{d}'_2$, put an oriented edge
from $\mathbf{d}_1\otimes\mathbf{d}_2$ to 
$\mathbf{d}'_1\otimes\mathbf{d}'_2$ provided that there is 
a partition diagram $\tilde{\mathbf{d}}$ which, when applied
to $\mathbf{d}_1\otimes\mathbf{d}_2$, outputs 
$\mathbf{d}'_1\otimes\mathbf{d}'_2$, up to a non-zero scalar.
Note that $\Gamma$ does not depend on $x\in\{a,b\}$, since we assume $x\neq 0$.
Also note that all vertices have loops as we can choose
$\tilde{\mathbf{d}}$ to be the identity partition diagram.

Since $\Gamma$ is finite, we can choose a filtration of 
$\Gamma$ of the form
\begin{displaymath}
\varnothing =\Gamma_0\subset \Gamma_1\subset
\dots\subset \Gamma_k=\Gamma
\end{displaymath}
by non-empty full subgraphs such that the following two conditions
are satisfied, for each $i$:
\begin{itemize}
\item If $\alpha$ is an edge from $u$ to $v$ and $u\in\Gamma_i$,
then $v\in \Gamma_i$.
\item The full subgraph of $\Gamma_i$ whose vertices are all vertices
outside $\Gamma_{i-1}$ is strongly connected.
\end{itemize}
Then the linearization $\Bbbk[\Gamma_i]$ of the vertex set of each
$\Gamma_i$ is a submodule of $X\otimes_{\Bbbk} Y$, so we have the
induced filtration of $X\otimes_{\Bbbk} Y$ by submodules
\begin{displaymath}
0\subset  \Bbbk[\Gamma_1]\subset  \Bbbk[\Gamma_2]\subset\dots
\subset X\otimes_{\Bbbk} Y
\end{displaymath}
and this filtration, by construction, does not depend on $x$.

We claim that each $\mathcal{P}_n(ab)$-module
$N_i:=\Bbbk[\Gamma_i]/\Bbbk[\Gamma_{i-1}]$ has 
a filtration by standard modules. Let $j$ be the minimal
rank for which there is a partition diagram that does
not annihilate $\Bbbk[\Gamma_i]/\Bbbk[\Gamma_{i-1}]$.
Choose some idempotent partition diagram 
$\tilde{\mathbf{d}}$ of rank $j$ and fix some
$\mathbf{d}_1\otimes\mathbf{d}_2\in \Gamma_i\setminus 
\Gamma_{i-1}$ which is fixed by $\tilde{\mathbf{d}}$
(this exists as $\tilde{\mathbf{d}}$ is an idempotent
and it does not annihilate $N_i$). Applying the left
cell of $\tilde{\mathbf{d}}$ to 
$\mathbf{d}_1\otimes\mathbf{d}_2$ gives a set of vertices
which is invariant under the action of all partition
diagrams (modulo $\Gamma_{i-1}$). Hence this must 
coincide with the whole of $\Gamma_i\setminus 
\Gamma_{i-1}$ due to strong connectivity of the latter.
In reality, we just proved that $N_i$ is a quotient
of $\Bbbk[\mathbf{L}_{\tilde{\mathbf{d}}}]$ since we have
a natural surjective homomorphism from the letter to the
former which sends $\tilde{\mathbf{d}}'\in \mathbf{L}_{\tilde{\mathbf{d}}}$
to $\tilde{\mathbf{d}}'\cdot (\mathbf{d}_1\otimes\mathbf{d}_2)$.

Let $G$ be the $\mathcal{H}$-class
of $\tilde{\mathbf{d}}$. Then $\mathbf{L}_{\tilde{\mathbf{d}}}$ is a
free right $G$-act. Let $H$ be the stabilizer of
$\mathbf{d}_1\otimes\mathbf{d}_2$ in $G$. Then the above gives
a surjection from the module
\begin{equation}\label{eq-abn-5}
\Bbbk[\mathbf{L}_{\tilde{\mathbf{d}}}]\otimes_{\Bbbk[G]}\Bbbk[G/H]
\end{equation}
to $N_i$. Note that the module \eqref{eq-abn-5} has a filtration
by standard modules by construction.

Everything above is independent of $x$, as long as $x\neq 0$. For 
generic $x$, standard modules are simple and hence the surjection
in the previous paragraph is, in fact an isomorphism between
the module \eqref{eq-abn-5} and $N_i$. But then the fact that this 
surjection is an isomorphism does not depend on $x$, completing the proof.
\end{proof}

\begin{remark}\label{rem7.4.1-5}
{\em
Analogues of Proposition~\ref{prop7.4.1} are true, with the 
same proof, for other classical diagram subalgebras of the
partition algebra, in particular, for the Brauer algebra,
the partial Brauer algebra, the Temperley-Lieb algebra etc.
}
\end{remark}

\begin{remark}\label{rem7.4.1-7}
{\em
For partition algebras, natural analogues of 
Littlewood--Richardson coefficient, that is coefficients which
describe multiplicities of standard modules in the induced
external tensor product of standard modules, are given by 
reduced Kronecker coefficient. This can be found in 
\cite[Theorem 4.3]{BDO} and in  \cite[Theorem~3.11]{BV}.
}
\end{remark}

From the theory of quasi-hereditary algebras, we know that the
multiplicity of a standard subquotient in a standard filtration 
of some module that admits a standard filtration
does not depend on the choice of such a  filtration.

Let us now go back to the setup as described in Subsection~\ref{s7.1}.

\begin{corollary}\label{cor7.4.1}
Let $\lambda\vdash k$, $\mu\vdash l$ and $\nu\vdash kl$.
Then, for 
\begin{displaymath}
n\geq \max(kl-\nu_1,k-\lambda_1,l-\mu_1), 
\end{displaymath}
the number 
$\mathbf{b}^\nu_{\lambda,\mu}$ coincides with the multiplicity 
of the $\mathcal{P}_n(kl)$-module $\mathbf{N}^{\nu_{>1}}$ in 
the tensor product of the $\mathcal{P}_n(k)$-module
$\mathbf{N}^{\lambda_{>1}}$ and the $\mathcal{P}_n(l)$-module
$\mathbf{N}^{\mu_{>1}}$.
\end{corollary}

\begin{proof}
By adjunction, the restriction from $S_{kl}$ to $S_k\times S_l$
on the left hand side of 
$(\Bbbk^k\otimes_{\Bbbk}\Bbbk^l)^{\otimes n}$
corresponds to induction from the tensor product of
$\mathcal{P}_n(k)$ and $\mathcal{P}_n(l)$ to
$\mathcal{P}_n(kl)$ on the right hand side.
Due to our assumption  $n\geq \max(kl-\nu_1,k-\lambda_1,l-\mu_1)$,
we can apply the bimodule decomposition in the Schur-Weyl duality.
As the connection from the left to the right hand side of 
the Schur-Weyl duality maps a partition $\gamma$ to $\gamma_{>1}$,
the claim follows. 
\end{proof}

\subsection{Stability phenomenon for $\mathbf{b}^\gamma_{\mu,\nu}$}\label{s7.5}

For a partition $\lambda\vdash n$ and a positive integer $a$, we
denote by $\lambda^{(a)}$ the partition of $n+a$ obtained from
$\lambda$ by increasing $\lambda_1$ to $\lambda_1+a$.
The following claim was also proved in \cite[Theorem~4.2]{Ryba}
by completely different arguments.

\begin{theorem}\label{s7.5.1}
Let $\lambda\vdash k$, $\mu\vdash l$ and $\nu\vdash kl$.
Then, for all $a\gg 0$, the value $\mathbf{b}^{\nu^{(la)}}_{\lambda^{(a)},\mu}$ is constant.
\end{theorem}

\begin{proof}
Let $n\geq \max(kl-\nu_1,k-\lambda_1,l-\mu_1)$ and note that
\begin{displaymath}
\max(kl-\nu_1,k-\lambda_1,l-\mu_1)=
\max((kl+la)-\nu^{(la)}_1,(k+a)-\lambda^{(a)}_1,l-\nu_1),
\end{displaymath}
for all $a>0$. This allows us to 
use the interpretation of $\mathbf{b}^{\nu^{(la)}}_{\lambda^{(a)},\mu}$
given by Corollary~\ref{cor7.4.1}. We have
$\nu^{(la)}_{>1}=\nu_{>1}$ and $\lambda^{(a)}_{>1}=\lambda_{>1}$, for all $a>0$.
In particular, on the partition algebra side, all involved modules have
exactly the same corresponding indices, regardless of $a>0$.

Next we note that, for all $a\gg 0$, both algebras 
$\mathcal{P}_n((k+a)l)$ and $\mathcal{P}_n(k+a)$ are
semi-simple, moreover, each simple module $\mathbf{N}^\gamma$ for
each of these algebras coincides with the corresponding 
$\mathbf{C}^\gamma$. 

The multiplicity of a simple module over a semi-simple finite dimensional
algebra can be computed as the rank of a certain primitive idempotent.
Generic primitive idempotents for partition algebras were constructed in
\cite{MW}. Outside of a finite number of values of the parameter,
each such idempotent is given as a linear combination of diagrams
whose coefficients are rational functions in the parameter of 
the partition algebra. We can take $a\gg 0$ such that the 
construction of these primitive idempotents applies (meaning that the denominators 
are non-zero).
If we fix such a primitive idempotent for $\mathcal{P}_n((k+a)l)$, 
then the action of this idempotent on the tensor product of some 
$\mathcal{P}_n(k+a)$-module $\mathbf{C}^\gamma$  (in which we fix the
standard diagram basis) with a 
fixed $\mathcal{P}_n(l)$-module is given by a matrix whose coefficients
are rational functions in $a$. Hence this matrix has maximal rank for all
but finitely many values of $a$. The claim follows.
\end{proof}

\subsection{Tensor product of cell modules for partition algebras}\label{s7.6}

Unlike the cases of the (dual) symmetric inverse semigroups, tensor product 
of cell modules for partition algebras does not have to have a filtration 
whose subquotients are cell modules. In this subsection we give an
explicit example illustrating how this property fails already for $n=2$.

The algebra $\mathcal{P}_2(x)$ has dimension $15$ (the bell number $\mathbf{B}_4$).
As usual, we assume $x\neq 0$.
The possible ranks for partition diagrams are $0,1$ and $2$.
We have one left cell $\mathcal{L}_2$ of rank two, consisting of the
following two diagrams:
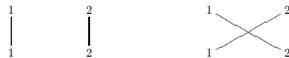
\begin{displaymath}
\resizebox{12mm}{!}{\xymatrix{1\ar@{-}[d]&&2\ar@{-}[d]\\1&&2}}\qquad\qquad
\resizebox{12mm}{!}{\xymatrix{1\ar@{-}[drr]&&2\ar@{-}[dll]\\1&&2}}
\end{displaymath}
We have three left cell $\mathcal{L}_1^{(i)}$, where $i=1,2,3$, 
of rank one, each containing three elements, given by the following rows:
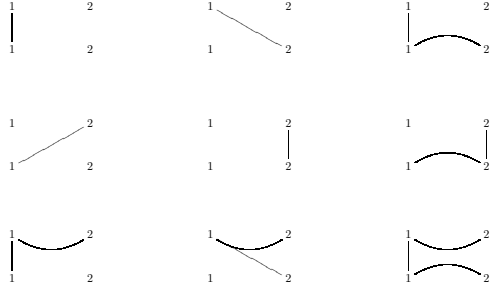
\begin{displaymath}
\resizebox{12mm}{!}{\xymatrix{1\ar@{-}[d]&&2\\1&&2}}\qquad\qquad
\resizebox{12mm}{!}{\xymatrix{1\ar@{-}[drr]&&2\\1&&2}}\qquad\qquad
\resizebox{12mm}{!}{\xymatrix{1\ar@{-}[d]&&2\\1\ar@/^1pc/@{-}[rr]&&2}}
\end{displaymath}

\begin{displaymath}
\resizebox{12mm}{!}{\xymatrix{1&&2\ar@{-}[dll]\\1&&2}}\qquad\qquad
\resizebox{12mm}{!}{\xymatrix{1&&2\ar@{-}[d]\\1&&2}}\qquad\qquad
\resizebox{12mm}{!}{\xymatrix{1&&2\ar@{-}[d]\\1\ar@/^1pc/@{-}[rr]&&2}}
\end{displaymath}

\begin{displaymath}
\resizebox{12mm}{!}{\xymatrix{1\ar@{-}[d]&&2\ar@/^1pc/@{-}[ll]\\1&&2}}\qquad\qquad
\resizebox{12mm}{!}{\xymatrix{1\ar@{-}[drr]&&2\ar@/^1pc/@{-}[ll]\\1&&2}}\qquad\qquad
\resizebox{12mm}{!}{\xymatrix{1\ar@{-}[d]&&2\ar@/^1pc/@{-}[ll]\\1\ar@/^1pc/@{-}[rr]&&2}}
\end{displaymath}

Finally, we have two left cell $\mathcal{L}_0^{(i)}$, where $i=1,2$, 
of rank zero, each containing two elements, given by the following rows:
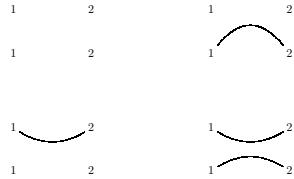
\begin{displaymath}
\resizebox{12mm}{!}{\xymatrix{1&&2\\1&&2}}\qquad\qquad
\resizebox{12mm}{!}{\xymatrix{1&&2\\1\ar@/^2pc/@{-}[rr]&&2}}
\end{displaymath}

\begin{displaymath}
\resizebox{12mm}{!}{\xymatrix{1&&2\ar@/^1pc/@{-}[ll]\\1&&2}}\qquad\qquad
\resizebox{12mm}{!}{\xymatrix{1&&2\ar@/^1pc/@{-}[ll]\\1\ar@/^1pc/@{-}[rr]&&2}}
\end{displaymath}

In particular, we have, up to isomorphism, three cell modules
$\Bbbk[\mathcal{L}_2]$, $\Bbbk[\mathcal{L}_1^{(1)}]$ and
$\Bbbk[\mathcal{L}_0^{(1)}]$, or respective dimensions $2$, $3$ and $2$.

Now we claim that $\Bbbk[\mathcal{L}_1^{(1)}]\otimes_\Bbbk 
\Bbbk[\mathcal{L}_0^{(1)}]$ does not have a cell filtration.
Indeed, this module has dimension six with the following basis:
\begin{displaymath}
v_1:= 
\resizebox{7mm}{!}{\xymatrix{1\ar@{-}[d]&&2\\1&&2}}\otimes
\resizebox{7mm}{!}{\xymatrix{1&&2\\1&&2}}
\qquad\qquad
v_2:= 
\resizebox{7mm}{!}{\xymatrix{1\ar@{-}[drr]&&2\\1&&2}}\otimes
\resizebox{7mm}{!}{\xymatrix{1&&2\\1&&2}}
\qquad\qquad
v_3:= 
\resizebox{7mm}{!}{\xymatrix{1\ar@{-}[d]&&2\\1\ar@/^1pc/@{-}[rr]&&2}}\otimes
\resizebox{7mm}{!}{\xymatrix{1&&2\\1&&2}}
\end{displaymath}
\begin{displaymath}
w_1:= 
\resizebox{7mm}{!}{\xymatrix{1\ar@{-}[d]&&2\\1&&2}}\otimes
\resizebox{7mm}{!}{\xymatrix{1&&2\\1\ar@/^2pc/@{-}[rr]&&2}}
\qquad\qquad
w_2:= 
\resizebox{7mm}{!}{\xymatrix{1\ar@{-}[drr]&&2\\1&&2}}\otimes
\resizebox{7mm}{!}{\xymatrix{1&&2\\1\ar@/^2pc/@{-}[rr]&&2}}
\qquad\qquad
w_3:= 
\resizebox{7mm}{!}{\xymatrix{1\ar@{-}[d]&&2\\1\ar@/^1pc/@{-}[rr]&&2}}\otimes
\resizebox{7mm}{!}{\xymatrix{1&&2\\1\ar@/^2pc/@{-}[rr]&&2}}
\end{displaymath}

It is easy to check that $v_1$, $v_2$ and $w_3$ span in 
$\Bbbk[\mathcal{L}_1^{(1)}]\otimes_\Bbbk 
\Bbbk[\mathcal{L}_0^{(1)}]$ a submodule isomorphic to 
$\Bbbk[\mathcal{L}_1^{(1)}]$. In the quotient, 
$w_1$ and $w_2$ span a submodule isomorphic to $\Bbbk[\mathcal{L}_2]$,
while $v_3$ only spans a submodule of $\Bbbk[\mathcal{L}_2]$, as the
latter has dimension two and we only have one vector left.


\noindent
V.~M.: Department of Mathematics, Uppsala University, Box. 480,
SE-75106, Uppsala, SWEDEN, email: {\tt mazor\symbol{64}math.uu.se}

\noindent
S.~S.: Department of Mathematics, Indian Institute of Technology, 
Dharwad, Karnataka, 580011, INDIA, email: {\tt maths.shraddha\symbol{64}gmail.com}

\end{document}